\documentclass[a4paper,reqno]{amsart}

\usepackage{mlmodern}
\usepackage[T1]{fontenc}
\usepackage[utf8]{inputenc}
\usepackage[UKenglish]{babel}

\DeclareFontFamily{OMX}{mlmex}{}
\DeclareFontShape{OMX}{mlmex}{m}{n}{<->mlmex10}{}

\usepackage[babel]{microtype}

\usepackage{mathtools}
\usepackage{amssymb}
\usepackage[psdextra]{hyperref}

\usepackage{tikz}
\tikzset{
    every picture/.style = {
        line width = 1pt,
        line cap   = round,
    },
    every plot/.style = {
        smooth,
        samples = 100,
    },
}

\theoremstyle{plain}
\newtheorem{theorem}             {Theorem}    [section]
\newtheorem{corollary}  [theorem]{Corollary}
\newtheorem{lemma}      [theorem]{Lemma}
\newtheorem{proposition}[theorem]{Proposition}

\theoremstyle{definition}
\newtheorem{definition} [theorem]{Definition}
\newtheorem{remark}     [theorem]{Remark}

\newcommand\R{\mathbf R}

\newcommand\Z{\mathbf Z}
\newcommand\N{\mathbf N}

\newcommand\T{\mathbf T}
\renewcommand\S{\mathbf S}

\newcommand\x{\mathbf x}

\newcommand\blank{{\mkern2mu\cdot\mkern2mu}}

\ExplSyntaxOn
\newcommand\diffop{\mathop{}\!\mathrm d}
\NewDocumentCommand \diff { m }
{
    \clist_map_inline:nn { #1 } { \diffop ##1 }
}
\ExplSyntaxOff

\newcommand\Hausdorff{\mathcal H}
\newcommand\dH{\diff\Hausdorff}

\newcommand\weakstarto{\stackrel{*}{\rightharpoonup}}
\newcommand\upto\uparrow
\newcommand\downto\downarrow

\newcommand\symdiff{\mathbin\triangle}

\DeclareMathOperator\Per{Per}
\DeclareMathOperator\id{id}
\renewcommand\div{\operatorname{div}}

\DeclareMathOperator*\Gammalim{\Gamma-lim}
\DeclareMathOperator*\argmin{arg\,min}

\DeclarePairedDelimiter\abs\lvert\rvert
\DeclarePairedDelimiter\norm\lVert\rVert

\DeclarePairedDelimiter\mean\langle\rangle
\DeclarePairedDelimiterX\scalar[2]\langle\rangle{ #1 , #2 }

\usepackage{orcidlink}

\newcommand\dash{\unskip\thinspace\textemdash\penalty\exhyphenpenalty\hskip.16667em\relax\ignorespaces} 
\usepackage{graphicx}

\hypersetup{
    pdftitle    = {A capillary problem, its dimension reduction, and its phase-field approximation},
    pdfauthor   = {Patrick Dondl, Lars Pastewka, Luciano Sciaraffia, and Yizhen Wang},
    pdfkeywords = {Capillarity, adhesion, Γ-convergence, dimension reduction, pinning, hysteresis},
    colorlinks  = true,
    linkcolor   = red,
    citecolor   = blue,
}

\usepackage[alphabetic, initials, msc-links]{amsrefs}

\numberwithin{equation}{section}

\begin{document}

\title[Dimension reduction of a capillary problem]{A capillary problem, its dimension reduction, and its phase-field approximation}

\author[Patrick Dondl]{Patrick Dondl$\,^*$ \orcidlink{0000-0003-3035-7230}}
\address{
    Abteilung für Angewandte Mathematik \\
    Albert-Ludwigs-Universität Freiburg \\
    Hermann-Herder-Straße 10 \\
    79104 Freiburg im Breisgau \\
    Germany.
}
\email{patrick.dondl@mathematik.uni-freiburg.de}
\thanks{$^*$Corresponding author.}

\author[Lars Pastewka]{Lars Pastewka \orcidlink{0000-0001-8351-7336}}
\address{
    Institut für Mikrosystemtechnik \\
    Albert-Ludwigs-Universität Freiburg \\
    Georges-Köhler-Allee 103 \\
    79110 Freiburg im Breisgau \\
    Germany.
}
\email{lars.pastewka@imtek.uni-freiburg.de}

\author[Luciano Sciaraffia]{Luciano Sciaraffia \orcidlink{0009-0000-0142-2246}}
\address{
    Abteilung für Angewandte Mathematik \\
    Albert-Ludwigs-Universität Freiburg \\
    Hermann-Herder-Straße 10 \\
    79104 Freiburg im Breisgau \\
    Germany.
}
\email{luciano.sciaraffia@mathematik.uni-freiburg.de}

\author[Yizhen Wang]{Yizhen Wang \orcidlink{0009-0007-7414-2633}}
\address{
    Institut für Mikrosystemtechnik \\
    Albert-Ludwigs-Universität Freiburg \\
    Georges-Köhler-Allee 103 \\
    79110 Freiburg im Breisgau \\
    Germany.
}
\email{yizhen.wang@imtek.uni-freiburg.de}


\begin{abstract}
    We study the behaviour of a given volume of liquid confined between two rough solid plates. When the separation between the plates is small relative to the liquid volume, capillary bridges are expected to form, which minimise Gauss's capillary energy locally. We derive a $\Gamma$-expansion for the energy as the plate separation approaches zero, yielding a dimensionally reduced problem in terms of the wetted regions on the plates. At leading order, the energy is determined by the area of the wetted regions, while the next order term is given by their perimeter, weighted by appropriate functions of the relative adhesion coefficients. This provides a framework for a subsequent phase-field approximation, which is employed in numerical simulations to study the evolution of the droplets under the normal movement of the plates. The theory also gives justification to models for capillarity that assume that liquids fill up the rough topography to the Kelvin radius.
\end{abstract}

\maketitle

\setcounter{tocdepth}{1}
\tableofcontents

\section{Introduction}

Water is ubiquitous \dash typically present as vapour in the air that surrounds us. This vapour condenses on surfaces to form liquid droplets. When condensation occurs between two surfaces in close proximity, the resulting liquid bridge can effectively bond the surfaces together \dash a phenomenon known as capillary adhesion. Capillary adhesion plays a critical role across a wide range of natural and technological contexts. In micro- and nanoelectromechanical systems, capillary forces are a primary cause of stiction, often leading to device failure~\cite{vanspengen:2003}. The cohesion of wet granular materials \dash from sand to suspensions \dash is governed by liquid bridges between individual grains~\cites{moller-bonn:2007, koos-willenbacher:2011, pakpour-habibi-moller-ea:2012}. In biology, capillary forces contribute to the adhesion mechanisms of insects and small animals~\cites{huber-mantz-spolenak-ea:2005, autumn:2007}. Even at macroscopic scales, capillary adhesion affects the friction between wet surfaces~\cite{peng-hsia-woutersen-ea:2022}. Because capillary condensation is thermodynamically favoured below saturation, these effects arise spontaneously under ambient conditions and become increasingly pronounced as relative humidity rises~\cite{autumn:2007}.

The classical treatment of capillary adhesion, as presented in standard textbooks~\cite{israelachvili:2011}, centres around single capillary bridges that form between the contacting interfaces. Yet nearly all natural and engineered surfaces are rough~\cite{persson-albohr-tartaglino-ea:2005}. Surface roughness fundamentally alters the geometry of contact: rather than a single, well-defined meniscus, multiple capillary bridges nucleate at discrete locations where the local gap is sufficiently small for condensation to occur~\cite{desouza-brinkmann-mohrdieck-ea:2008}. While the literature acknowledges this complexity, only approximate theories for quantifying how surface roughness affects the effective capillary force between two contacting bodies have been developed to date~\cites{persson:2008,peng-hsia-woutersen-ea:2022}.

Existing models for the effective capillary force between rough surfaces rest on an \emph{ad~hoc} filling-up ansatz: the liquid is assumed to occupy those regions where the local gap between the surfaces falls below a threshold set by the Kelvin length $d_\mathrm{K} = r_\mathrm{K}(\cos\theta_0 + \cos\theta_1)$, with~$r_\mathrm{K}$ the Kelvin radius of a meniscus in thermodynamic equilibrium with the surrounding vapour~\cites{persson:2008, peng-hsia-woutersen-ea:2022}. Geometrically, the wetted region is then a sublevel set of the gap profile, bounded by one of its level sets: points where the interfacial separation falls below~$d_\mathrm{K}$ are flooded. Under this assumption the wetted patches are slit islands~\cite{mandelbrot:1982}, coinciding with the contact patches predicted by bearing-area models~\cite{ramisetti-campana-anciaux-ea:2011}. This ansatz, however, is imposed rather than derived: it fixes the wetted region purely by the gap geometry and leaves no room for the liquid--vapour interfacial energy to act. One aim of the present paper is to show that it is, in fact, the leading-order consequence of energy minimisation, and to identify the correction that the interfacial energy contributes at the next order.

From a mathematical perspective, the analysis of capillary phenomena builds on the theory of sets of finite perimeter and the associated variational framework, as developed in the foundational works of De~Giorgi and later systematised by Giusti~\cite{giusti:1984} and Maggi~\cite{maggi:2012}. The equilibrium configurations of liquid drops are characterised as minimisers of Gauss's free energy functional, which balances the liquid-vapour interfacial energy against wetting contributions along the solid boundary. The mathematical study of such capillary problems, including existence, regularity, and qualitative properties of minimisers, has a rich history surveyed comprehensively by Finn~\cite{finn:1986} \dash see also~\cite{maggi:2012}*{Chapter 19} for an introductory treatment of the subject.

The rigorous passage from a full-dimensional model to a reduced one is naturally formulated in terms of $\Gamma$-convergence, a variational notion of convergence introduced by De~Giorgi and Franzoni~\cite{degiorgi-franzoni:1975} that ensures convergence of minimisers and infima \dash see Dal~Maso~\cite{dalmaso:1993} for a comprehensive treatment. Dimension reduction via $\Gamma$-convergence has proven particularly powerful in the mechanics of thin structures, where it has been employed to derive membrane, plate, and rod theories from three-dimensional elasticity~\cites{acerbi-buttazzo-percivale:1991, ledret-raoult:1995, mora-muller:2003, friesecke-james-muller:2006}. The present work adapts this methodology to capillary problems: as the plate separation tends to zero, the three-dimensional free boundary problem reduces to a two-dimensional model posed on the wetted region. Ordinary $\Gamma$-convergence $F_\varepsilon \to F^0$, however, captures only the leading order: $F^0$ here admits a whole family of minimisers, so it alone neither distinguishes which of them arises as a limit of minimisers of~$F_\varepsilon$, nor resolves the rate at which the minimal values $\min F_\varepsilon$ approach $\min F^0$. To capture this finer information we employ the notion of \emph{$\mathit\Gamma$-development} (or \emph{$\mathit\Gamma$-expansion}), introduced by Anzellotti and Baldo~\cite{anzellotti-baldo:1993} and later applied to the elasticity theory of rods and plates~\cite{anzellotti-baldo-percivale:1994}, which provides a systematic asymptotic development $F_\varepsilon = F^0 + \varepsilon F^1 + o(\varepsilon)$, where $F^1$ is the $\Gamma$-limit of the rescaled functionals $\varepsilon^{-1}(F_\varepsilon - \min F^0)$ restricted to the minimisers of~$F^0$. A companion notion is that of \emph{$\mathit\Gamma$-equivalence} in the sense of Braides and Truskinovsky~\cite{braides-truskinovsky:2008}, which replaces the asymptotic formula $F^0 + \varepsilon F^1$ by a simpler $\varepsilon$-dependent family~$G_\varepsilon$ whose direct minimisation reproduces $F_\varepsilon$ to the prescribed order.

Our results are as follows. At leading order, the $\Gamma$-limit~$F^0$ (Theorem~\ref{thm:Gammalim0}) reduces the three-dimensional capillary problem to an optimal wetting problem on the cross-section, and the characterisation of its minimisers (Proposition~\ref{prop:E0}) shows that the liquid occupies exactly the sublevel set~$\{\overline\sigma < \lambda^* g\}$ of the gap~$g$. This is precisely the filling-up ansatz of the physics literature, now \emph{derived} rather than postulated: the threshold~$\lambda^*$ emerges as the Lagrange multiplier of the volume constraint and plays the role of the chemical potential that the Kelvin radius fixes in~\cites{persson:2008, peng-hsia-woutersen-ea:2022}, and for homogeneous plates the wetted region is a genuine level set of~$g$. At the next order, the $\Gamma$-development yields a weighted-perimeter functional~$F^1$ (Theorem~\ref{thm:Gammalim1}) and a dimensionally reduced family of functionals equivalent to~$F_\varepsilon$, together with an expansion of the minimal energy. This first-order term measures the deviation of the true wetted region from the sharp level set: its weight is largest at the neutral contact angle~$\pi/2$ \dash where the filling-up property disappears entirely \dash and smallest in the perfectly wetting and perfectly drying limits ($\sigma \to \pm 1$, contact angles $0$ and~$\pi$), so that the level-set picture is most accurate near those extremes. Finally, when the limiting interface is irregular the development may break down; we quantify this through the Minkowski dimension~$\alpha$ of the limiting boundary (Theorem~\ref{thm:irreg-mins}), exhibiting fractal interfaces for which the minimiser scales as~$\varepsilon^{d-\alpha}$.

We complement the analysis with numerical experiments based on a phase-field approximation of the reduced functional, in the spirit of Modica and Mortola~\cites{modica-mortola:1977, modica:1987a}, whose $\Gamma$-convergence to the perimeter functional provides a robust framework for tracking the evolving droplet geometry. The simulations confirm the picture above: the computed minimisers closely track the level-set filling predicted by Proposition~\ref{prop:E0}, with the agreement sharpest at the extreme contact angles and the perimeter term visibly smoothing and merging islands as the neutral angle is approached. They also yield the macroscopic capillary force as the plates are separated, which is attractive for hydrophilic interfaces, repulsive for hydrophobic ones, and smaller by some orders of magnitude at the neutral angle, and which exhibits hysteresis over approach--retraction cycles. Beyond illustrating the theory, these computations demonstrate that the reduced model can be solved on grids spanning the several orders of magnitude in length scale required to resolve realistic roughness \dash a regime out of reach for the full three-dimensional free boundary problem \dash and thereby provides the foundation for a quantitative model of capillary adhesion between rough surfaces. The irregular regime of Theorem~\ref{thm:irreg-mins}, where the limiting interface has infinite perimeter, remains an avenue for further work.

The remainder of the article is organised as follows. In Section~\ref{sec:setting} we recall Gauss's capillary model, formulate the problem in the thin domain $\Omega_\varepsilon$ between the two plates, rewrite it on the fixed cylinder $\Omega = \omega \times (0,1)$, and state our main results. Section~\ref{sec:proofs} is devoted to the proofs of these results. In Section~\ref{sec:numerics}, we develop a phase-field approximation of the reduced functional and present numerical experiments illustrating the evolution of the droplets under normal motion of the plates. Finally, for the convenience of the reader, two appendices collect some of the background material used throughout.

\section{Mathematical setting and main results}\label{sec:setting}

In Gauss's capillary model, stable configurations of a liquid drop of fixed volume are determined by minimising a total surface energy. This energy balances the interactions across the liquid-vapour~(lv), liquid-solid~(ls), and solid-vapour~(sv) interfaces:
\begin{equation}\label{eq:explicit-energy}
    \gamma_\mathrm{lv} \Hausdorff^d(\Sigma_\mathrm{lv}) + \int_{\Sigma_\mathrm{ls}} \gamma_\mathrm{ls}(\x) \dH^d(\x) + \int_{\Sigma_\mathrm{sv}} \gamma_\mathrm{sv}(\x) \dH^d(\x) .
\end{equation}
Here, $\Hausdorff^d$ denotes the $d$-dimensional Hausdorff measure (representing the interfacial area). The strictly positive quantities~$\gamma$ are called \emph{surface tensions.} It is a reasonable physical assumption to treat $\gamma_\mathrm{lv}$ as a constant, whereas $\gamma_\mathrm{ls}$ and~$\gamma_\mathrm{sv}$ are allowed to depend on the spatial coordinate~$\x$ to account for the heterogeneity of the fixed solid phase.

We model the solid container as a bounded open set $\Omega \subset \R^{d+1}$ ($d \geq 1$) and the liquid drop as a set of finite perimeter $E \subset \Omega$. The relevant interfaces can then be expressed in terms of the reduced boundary $\partial^*\!E$ and the container boundary $\partial\Omega$:
\[
    \Sigma_\mathrm{lv} = \partial^*\!E \cap \Omega , \quad \Sigma_\mathrm{ls} = \partial^*\!E \cap \partial\Omega , \quad \Sigma_\mathrm{sv} = \partial\Omega \setminus \partial^*\!E .
\]
Observe that the total surface tension over the solid boundary is a fixed constant, independent of the drop's configuration:
\[
    \int_{\Sigma_\mathrm{sv}} \gamma_\mathrm{sv}(\x) \dH^d(\x) + \int_{\Sigma_\mathrm{ls}} \gamma_\mathrm{sv}(\x) \dH^d(\x) = \int_{\partial\Omega} \gamma_\mathrm{sv}(\x) \dH^d(\x) .
\]
By substituting this relation into the initial energy and dividing by the constant~$\gamma_\mathrm{lv}$, the problem is equivalent, after renormalisation, to minimising the functional
\[
    \Per(E,\Omega) + \int_{\partial\Omega} \chi_E(\x) \sigma( \x ) \dH^d( \x )
\]
among all subsets $E \subset \Omega$ with a prescribed volume. Here, $\Per(\blank,\Omega)$ is the relative perimeter functional in $\Omega$, given by
\[
    \Per(E,\Omega) \coloneq \sup_{\substack{\mathbf X \in C_0^\infty(\Omega;\R^{d+1}) \\ \norm{ \mathbf X }_{L^\infty} \leq 1}} \int_E \div \mathbf X \diff\x = \Hausdorff^d(\partial^*\!E \cap \Omega) ,
\]
and $\chi_E$ denotes the trace of the characteristic function of~$E$ on~$\partial\Omega$. Finally, the function
\[
    \sigma(\x) \coloneq \frac{ \gamma_\mathrm{ls}(\x) - \gamma_\mathrm{sv}(\x) }{ \gamma_\mathrm{lv} } , \quad \x \in \partial\Omega ,
\]
is called the \emph{relative adhesion coefficient,} and we assume the \emph{wetting condition} $\norm{ \sigma }_{L^\infty} < 1$.

In our situation the container is given by the region bounded by two rough solid plates, which are order-$\varepsilon$ close to one another, for $\varepsilon > 0$ the scale parameter. We represent these plates as graphs of two functions $h_j \in W^{1,\infty}(\omega)$ ($j=0,1$), where $\omega \subset \R^d$ is an open set in a space of one dimension less. We may also treat the case with periodic boundary conditions, for which we identify $\omega = \T^d \coloneq \R^d / \Z^d$. For simplicity, we suppose that
\[
    \int_\omega h_j(x) \diff x = 0 , \quad j = 0,1 .
\]
The container is then the set
\[
    \Omega_\varepsilon \coloneq \bigl\{ \x = (x,y) \in \omega \times \R : \varepsilon h_0(x) < y < \varepsilon z + \varepsilon h_1(x) \bigr\} ,
\]
where $z > 0$ is the average separation of the plates. We define
\[
    g(x) \coloneq z + h_1(x) - h_0(x) , \quad x \in \omega ,
\]
as the \emph{gap} between the plates, and assume $g > 0$.

With this in mind, we may consider the relative adhesion coefficient of~$\Omega_\varepsilon$ as
\[
    \sigma(x,y) =
    \begin{cases}
        \sigma_0(x) , & \text{if}\quad y = \varepsilon h_0(x) ,                 \\
        \sigma_1(x) , & \text{if}\quad y = \varepsilon z + \varepsilon h_1(x) , \\
        0 ,           & \text{if}\quad x \in \partial\omega ,
    \end{cases}
\]
with $\sigma_j \in L^\infty(\omega)$ continuous functions ($j=0,1$). Note we are assuming that the relative adhesion coefficient is zero on the lateral walls of the container,
\[
    \{ (x,y) \in \partial\omega \times \R : \varepsilon h_0(x) < y < \varepsilon z + \varepsilon h_1(x) \} ,
\]
just to simplify the proofs and computations \dash it is otherwise possible to consider different values, independent of the height parameter~$y$, without changing the essence of the result. Indeed, the walls have height of order~$\varepsilon$, so their $\Hausdorff^d$-measure is~$O(\varepsilon)$. A non-zero adhesion coefficient~$\sigma_{\mathrm{wall}}$ therefore alters $F_\varepsilon$ by at most~$O(\varepsilon)$, leaving the $\Gamma$-limit~$F^0$ and its minimum value~$m_0$ unchanged, and contributing to the next term in the expansion $F^1$ only the fixed boundary term $\int_{\partial^*\!E_0 \cap \partial\omega} g \sigma_{\mathrm{wall}} \dH^{d-1}$, which can be carried through the proofs of Section~\ref{sec:proofs} without modification.

\begin{figure}
    \begin{tikzpicture}[
            font  = \footnotesize,
            scale = .8\linewidth/6cm,
            declare function = {
                hzero(\t) = .15 * (sin(\t*240 + 20) + .5*cos(\t*420));
                hone(\t)  = 1 + .15 * (cos(\t*300) + .5*sin(\t*540 + 40));
            }
        ]

        \path[use as bounding box] (0,-.8) rectangle (6,1.5);

        \draw[fill] (0,-.5) circle (.5pt) node[below] {$\partial\omega$} -- node[below] {$\omega \subset \R^d$} (6,-.5) circle (.5pt) node[below] {$\partial\omega$};

        \draw[fill=black!10] plot[domain=0:6] (\x, {hzero(\x)}) -- plot[domain=6:0] (\x, {hone(\x)}) -- cycle;

        \draw[line width=.5pt, dash dot, opacity=.5] (0, 0) -- (6, 0) node[right] {$y=0$} (0, 1) -- (6, 1) node[right] {$y=\varepsilon z$};

        \begin{scope}[line width=1.5pt]

            \clip plot[domain=0:6] (\x, {hzero(\x)}) -- plot[domain=6:0] (\x, {hone(\x)}) -- cycle;

            \draw[fill=black!25] (2.5, .5) circle (1.2) (5, .5) circle (.6) (.7, .6)  circle (.15);
            \draw[fill=black!10] (3, .4) circle (.2);

            \begin{scope}[line width=2pt]
                \clip (2.5, .5) circle (1.2);
                \draw[domain=0:6] plot (\x, {hzero(\x)}) plot (\x, {hone(\x)});
            \end{scope}

            \begin{scope}[line width=2pt]
                \clip (5, .5) circle (.6);
                \draw[domain=0:6] plot (\x, {hzero(\x)}) plot (\x, {hone(\x)});
            \end{scope}

        \end{scope}

        \node at (2.2, .5) {\Huge $E$};
        \node[anchor=south east] at (6, -.5) {$y = \varepsilon h_0(x)$};
        \node[anchor=north east] at (5.5, 1.5) {$y = \varepsilon z + \varepsilon h_1(x)$};
        \node[anchor=north west] at (.3, -.1) {$\sigma=\sigma_0(x)$};
        \node[anchor=south west] at (.42, 1.1) {$\sigma=\sigma_1(x)$};
        \node[anchor=south, rotate=90]  at (-.05, .5) {$\sigma=0$};
        \node[anchor=south, rotate=-90] at (6.05, .5) {$\sigma=0$};
        \draw[|<->|, line width=.5] (-.4, 1) -- node[left] {$\varepsilon z$} (-.4, 0);
        \draw[<->, line width=.5] (3.9, {hzero(3.9)}) -- node[right] {$\varepsilon g(x)$} (3.9, {hone(3.9)});

    \end{tikzpicture}
    \caption{Schematic view of the thin container $\Omega_\varepsilon \subset \R^{d+1}$ (light grey) constructed over a domain $\omega \subset \R^d$ and liquid $E$ (dark grey).}\label{fig:container}
\end{figure}
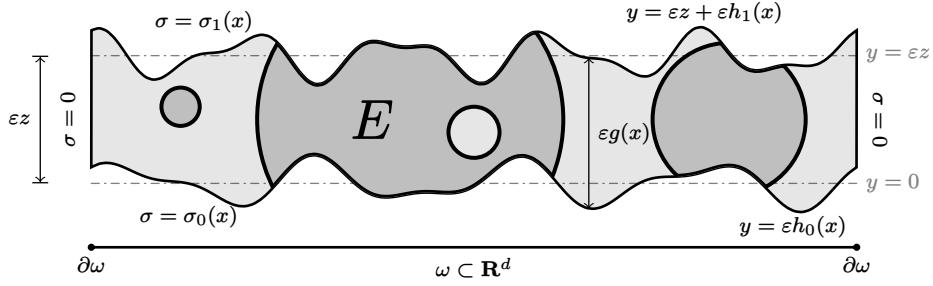

Therefore, the functional we are looking at is
\[
    F_\varepsilon(E) \coloneq \Per(E,\Omega_\varepsilon) + \int_{\partial\Omega_\varepsilon} \chi_E \sigma( \x ) \dH^d( \x ) ,
\]
and we are interested in approximating the minimisers of the variational problem
\begin{equation}\label{eq:minFe}
    m_\varepsilon \coloneq \min \{ F_\varepsilon(E) : E \subset \Omega_\varepsilon , \abs{ E } = \varepsilon a \} ,
\end{equation}
where $a \in (0,z \abs{ \omega })$ is a given fraction of the total volume of the container~$\Omega_\varepsilon$. Figure~\ref{fig:container} provides an illustration of this problem.

In order to take the limit as $\varepsilon \downto 0$, it is convenient to work on the fixed domain $\Omega \coloneq \omega \times (0,1)$. We achieve this by means of the diffeomorphism
\begin{equation}\label{eq:diffeo}
    \phi_\varepsilon \colon \omega \times (0,1) \to \Omega_\varepsilon , \quad \phi_\varepsilon(x,t) \coloneq \bigl( x, \varepsilon zt + \varepsilon h(x,t) \bigr) ,
\end{equation}
with
\[
    h(x,t) \coloneq (1-t) h_0(x) + t h_1(x)
\]
as the linear interpolation between the plates' defining functions. By a change of variables formula (cf.~\cite{maggi:2012}*{Theorem 17.1}), for any set $E \subset \overline{\Omega_\varepsilon}$ of finite perimeter and any continuous function~$f$, we have
\[
    \int_{\partial^*\!E} f \dH^d = \int_{\phi_\varepsilon^{-1}(\partial^*\!E)} ( f \circ \phi_\varepsilon ) \abs[\big]{ \operatorname{cof}\nabla\phi_\varepsilon \cdot \mathbf{n}_{\phi_\varepsilon^{-1}(\partial^*\!E)} } \dH^d .
\]
Here $\operatorname{cof}\nabla\phi_\varepsilon = (\det \nabla\phi_\varepsilon) (\nabla\phi_\varepsilon^{-1})^\top$ represents the \emph{cofactor matrix} of the gradient $\nabla\phi_\varepsilon$, while $\mathbf n_{\partial^*\!E}$ is the measure theoretic outward unit normal to a set~$E$. Thus, we can express the value of the functional~$F_\varepsilon$ at a set~$E$ in terms of its preimage $\phi_\varepsilon^{-1}(E)$ and $\mathbf{n}_{\phi_\varepsilon^{-1}(\partial^*\!E)} = (\nu,\eta)$ as
\begin{multline}\label{eq:Fe}
    F_\varepsilon(E) = \int_{\phi_\varepsilon^{-1}(\partial^*\!E) \cap \Omega} \sqrt{ \varepsilon^2 \abs{ g\nu - \eta\nabla_{\!x} h }^2 + \eta^2} \dH^d + {} \\
    \int_\omega \chi_{\phi_\varepsilon^{-1}(E)}(x,0) \sigma_0(x) \sqrt{ 1 + \varepsilon^2 \abs{ \nabla h_0(x) }^2 } \diff x + {} \\
    \int_\omega \chi_{\phi_\varepsilon^{-1}(E)}(x,1) \sigma_1(x) \sqrt{ 1 + \varepsilon^2 \abs{ \nabla h_1(x) }^2 } \diff x .
\end{multline}
On the other hand, the volume constraint simply becomes
\begin{equation}\label{eq:vol-constraint}
    \int_E \diff{x,y} = \varepsilon \int_{\phi_\varepsilon^{-1}(E)} g(x) \diff{x,t} = \varepsilon a .
\end{equation}
Therefore, we can redefine the functional~$F_\varepsilon$ using formula~\eqref{eq:Fe} on subsets $E \subset \Omega$, identifying each set with its image under the diffeomorphism~$\phi_\varepsilon$. The domain of~$F_\varepsilon$ is then given by all sets $E \subset \Omega$ of finite perimeter which satisfy the volume constraint in~\eqref{eq:vol-constraint}. As is customary, we set $F_\varepsilon(E) = \infty$ otherwise.

Our first result identifies the $\Gamma$-limit of the family $\{ F_\varepsilon \}_{\varepsilon>0}$. The topology we employ is the one given by the strong $L^1(\Omega)$-convergence, meaning that a sequence of sets $\{ E_\varepsilon \}_{\varepsilon>0}$ converges to a set $E \subset \Omega$ if, and only if,
\[
    \lim_{\varepsilon \downto 0} \int_\Omega{ \abs{ \chi_{E_\varepsilon} - \chi_E } }\diff\x = \lim_{\varepsilon \downto 0} | E_\varepsilon \symdiff E | = 0 .
\]

\begin{theorem}\label{thm:Gammalim0}
    The family of functionals $\{ F_\varepsilon \}_{\varepsilon>0}$ has the $\Gamma$-limit, with respect to the $L^1(\Omega)$-convergence as $\varepsilon \downto 0$, given by
    \[
        F^0(E) \coloneq \abs{ \mu_E } (\Omega) + \int_{\partial\Omega} \sigma \diff\mu_E ,
    \]
    with~$\mu_E$ the Radon measure characterised by the integration by parts in the vertical direction
    \[
        \int_E \frac{\partial\varphi}{\partial t} \diff{x,t} = \int_{\R^d \times \R} \varphi \diff\mu_E , \quad \varphi \in C^\infty_0(\R^d \times \R) .
    \]
    The domain of~$F^0$ consists of all sets $E \subset \Omega$ such that
    \[
        \sup_{\substack{\varphi \in C^\infty_0(\R^{d+1}) \\ \abs{ \varphi } \leq 1}} \int_E \frac{\partial\varphi}{\partial t} \diff{x,t} < \infty , \qquad \int_E g(x) \diff{x,t} = a .
    \]
\end{theorem}

The $\Gamma$-limit~$F^0$ already provides useful information about the minimisers of~$F_\varepsilon$. In fact, a complete characterisation of the minimisers of~$F^0$ is possible, as provided by the next proposition.

\begin{proposition}\label{prop:E0}
    The minimisers of~$F^0$ are sets of the form $E = E_0 \times (0,1)$, where $E_0 \subset \omega$ is any measurable subset which satisfies
    \begin{equation}\label{eq:minF0}
        \{ \overline\sigma < \lambda^*g \} \subset E_0 \subset \{ \overline\sigma \leq \lambda^*g \} , \qquad \int_{E_0} g(x) \diff x = a ,
    \end{equation}
    where $\lambda^*$ is the supremum of $\lambda \in \R$ such that the integral
    \[
        \int_{\{ \overline\sigma < \lambda g \}} g(x) \diff x < a ,
    \]
    and we denote by $\overline\sigma = \sigma_0 + \sigma_1$ the sum of the relative adhesion coefficients of each plate.
\end{proposition}

\begin{remark}
    Proposition~\ref{prop:E0} is the rigorous justification of the \emph{filling-up} ansatz used in the physics literature (cf.~the discussion in the introduction). It asserts that, at leading order, the liquid occupies exactly the sublevel set $\{ \overline\sigma < \lambda^* g \}$, that is, the region where the combined adhesion of the two plates is favourable enough relative to the local gap, the comparison being made through the ratio $\overline\sigma / g$. The threshold~$\lambda^*$ is the Lagrange multiplier associated with the volume constraint and plays the role of a chemical potential. It fixes the level $\{ \overline\sigma = \lambda^* g \}$ up to which the gap floods, in direct analogy with the flooding up to the Kelvin radius assumed in~\cites{persson:2008, peng-hsia-woutersen-ea:2022}. The direction in which the gap floods is governed by the sign of the adhesion. Recall that the relative adhesion coefficient is related to Young's contact angle~$\theta$ by $\sigma = -\cos\theta$, so that a plate is energetically favourable to wet ($\sigma < 0$) precisely when $\theta < \pi/2$, and unfavourable ($\sigma > 0$) when $\theta > \pi/2$; the neutral angle $\theta = \pi/2$ gives $\sigma = 0$. For homogeneous plates, where $\sigma_0$ and~$\sigma_1$ are constant, the combined coefficient~$\overline\sigma$ is constant and the wetted region becomes a level set of the gap~$g$ alone. If $\overline\sigma > 0$ (energetically unfavourable contact) the regions with the \emph{widest} gap fill first, whereas if $\overline\sigma < 0$ (favourable contact) the \emph{narrowest} regions fill first. Intuitively, per unit of liquid volume a narrow gap exposes more plate area, so favourable contact is maximised there while unfavourable contact is avoided. The latter case is the physically relevant capillary-condensation regime, in which the liquid fills the regions where the plates are closest, recovering the bearing-area picture.
\end{remark}

By one of the fundamental properties of $\Gamma$-convergence (see Theorem~\ref{thm:Gamma-min1}), as an immediate corollary we deduce the \emph{dimension reduction} for the minimum problem~\eqref{eq:minFe}.

\begin{corollary}[Dimension reduction]
    Let $E_\varepsilon \subset \Omega$ be a minimising set for~$F_\varepsilon$. Then, every limit point of $\{ E_\varepsilon \}_{\varepsilon>0}$ as $\varepsilon \downto 0$ is a set of the form $E_0 \times (0,1)$, with $E_0 \subset \omega$.
\end{corollary}

It follows from the definition of~$\lambda^*$ in Proposition~\ref{prop:E0} that the inequalities
\[
    \int_{\{\overline\sigma<\lambda^*g\}} g(x) \diff x \leq a \leq \int_{\{\overline\sigma \leq \lambda^*g\}} g(x) \diff x
\]
hold \dash indeed, since
\[
    \{ \overline\sigma < \lambda g \} = \bigcup_{k=1}^\infty{ \{ \overline\sigma < (\lambda - 1/k) g \} } , \qquad \{ \overline\sigma \leq \lambda g \} = \bigcap_{k=1}^\infty{ \{ \overline\sigma < (\lambda + 1/k) g \} }
\]
and so
\begin{multline*}
    \int_{ \{ \overline\sigma < \lambda^* g \} } g(x) \diff x = \lim_{k\to\infty} \int_{ \{ \overline\sigma < (\lambda^* - 1/k)g \} } g(x) \diff x \leq a \leq \\
    \lim_{k\to\infty} \int_{ \{ \overline\sigma < (\lambda^* + 1/k)g \} } g(x) \diff x = \int_{\{ \overline\sigma \leq \lambda^* g \} } g(x) \diff x .
\end{multline*}
In the case that either of them is achieved, the minimiser of~$F^0$ is unique \dash equal to $\{ \overline\sigma < \lambda^* g \}$ or $\{\overline\sigma \leq \lambda^*g\}$ respectively, up to sets of measure zero \dash and therefore the asymptotic shape of any family of minimisers $\{ E_\varepsilon \}_{\varepsilon>0}$ is completely determined. In other words, they are \emph{locked} (cf.~\cite{braides-truskinovsky:2008}*{Section~1.2}). However, when the inequalities are strict, there is some room for improvement in the portion of the equality set $\{ \overline\sigma = \lambda^* g \}$, which can be captured by the next term in the $\Gamma$-development.

With this in mind, we set
\[
    m_0 = \min \biggl\{ F^0(E) : \int_E g(x) \diff{x,t} = a \biggr\} .
\]
Note that, by Proposition~\ref{prop:E0}, $m_0$~can be expressed in terms of~$\lambda^*$ and the total volume~$a$ as
\[
    m_0 = \lambda^*a + \int_{\{ \overline\sigma < \lambda^* g \}} \overline\sigma(x) - \lambda^* g(x) \diff x .
\]
We may also identify a cylindrical set $E_0 \times (0,1) \subset \Omega$ with its cross section $E_0 \subset \omega$. In this sense, a sequence $E_\varepsilon \subset \Omega$ converges to $E_0 \subset \omega$ if $E_\varepsilon \to E_0 \times (0,1)$ in $L^1(\Omega)$.

Under the assumption that the obstacles $\{ \overline\sigma < \lambda^* g \}$ and $\{ \overline\sigma \leq \lambda^* g \}$ have finite perimeter, we find the next term at scale~$\varepsilon$.

\begin{theorem}\label{thm:Gammalim1}
    The family of rescaled functionals $\{ \varepsilon^{-1}(F_\varepsilon - m_0) \}_{\varepsilon>0}$ has the $\Gamma$-limit, with respect to the $L^1(\Omega)$-convergence as $\varepsilon \downto 0$, given by
    \[
        F^1(E) \coloneq \int_{\partial^*\!E \cap \omega} g(x) \biggl( \int_0^1 \sqrt{1-\sigma(x,t)^2} \diff t \biggr) \dH^{d-1}(x) ,
    \]
    where
    \[
        \sigma(x,t) \coloneq (t-1)\sigma_0(x) + t\sigma_1(x), \quad (x,t) \in \omega \times (0,1) ,
    \]
    is the linear interpolation between $-\sigma_0$ and~$\sigma_1$. The domain of~$F^1$ consists of all subsets~$E$ of~$\omega$ of finite perimeter which satisfy the constraints in~\eqref{eq:minF0}.
\end{theorem}

\begin{remark}
    As a functional, the zeroth-order $\Gamma$-limit~$F^0$ depends on neither the roughness~$h$ nor the gap~$g$ of the plates; the gap enters only through the volume constraint $\int_E g = a$, and hence through the location of the minimisers. The first-order term~$F^1$, by contrast, depends on the geometry of the plates only through the gap $g = z + h_1 - h_0$: for instance, two rough plates yield the same asymptotic behaviour as one flat plate facing a rough one whose roughness is $h = h_1 - h_0$. In particular, when the two plates are equal the gap $g \equiv z$ is constant, and the asymptotic behaviour is the same as that of flat plates.
\end{remark}

A consequence of the development is the $\Gamma$-equivalence (see Definition~\ref{def:Gamma-equiv}) with a family of dimensionally reduced functionals.

\begin{corollary}\label{cor:Ge}
    Suppose the sublevel sets $\{ \overline\sigma < \lambda^* g \}$ and $\{ \overline\sigma \leq \lambda^* g \}$ have finite perimeter in $\omega$. Then, the family of functionals $\{ F_\varepsilon \}_{\varepsilon>0}$ is equivalent at order $\varepsilon^\beta$, for all $\beta \in (0,1]$, to the family given by
    \[
        G_\varepsilon(E) \coloneq \int_E \overline\sigma(x) \diff x + \varepsilon \int_{\partial^*\!E \cap \omega} g(x) \biggl( \int_0^1 \sqrt{1-\sigma(x,t)^2} \diff t \biggr) \dH^{d-1}(x) .
    \]
    The domain of~$G_\varepsilon$ consists of all the sets $E \subset \omega$ of finite perimeter which satisfy
    \[
        \int_E g(x) \diff x = a .
    \]
    We identify $E \subset \omega$ with its cylindrical extension $E \times (0,1) \subset \Omega$.

    \begin{proof}
        It is easy to see that the family $\{ G_\varepsilon \}_{\varepsilon>0}$ has a $\Gamma$-limit $G^0 \coloneq \Gammalim_{\varepsilon \downto 0} G_\varepsilon$ with respect to $L^1(\omega)$-convergence defined by
        \[
            G^0(E) = \int_E \overline\sigma(x) \diff x ,
        \]
        with the same domain. Proposition~\ref{prop:E0} implies that the minimisers and the minimum value of~$G^0$ are the same as that of~$F^0$ (after identification between $E$ and $E \times (0,1)$). Therefore, from the definition of~$G_\varepsilon$ we deduce that
        \[
            \Gammalim_{\varepsilon \downto 0} \frac{ G_\varepsilon - m_0 }{ \varepsilon } = F^1 .
        \]
        It follows that $\{ G_\varepsilon \}_{\varepsilon>0}$ is equivalent to $\{ F_\varepsilon \}_{\varepsilon>0}$ at every order~$\varepsilon^\beta$, for $\beta \in (0,1]$.
    \end{proof}
\end{corollary}

Corollary~\ref{cor:Ge} establishes the expansion of the original capillary energy~$F_\varepsilon$, in the sense of a Braides--Truskinovsky~\cite{braides-truskinovsky:2008} $\Gamma$-equivalence, as the first-order Anzellotti--Baldo~\cite{anzellotti-baldo:1993} $\Gamma$-development
\begin{equation}\label{eq:Gamma-expansion}
    F_\varepsilon = G_\varepsilon + o(\varepsilon), \quad \text{where} \quad G_\varepsilon = F^0 + \varepsilon F^1
\end{equation}
when restricted to cylindrical sets.

By the compactness of minimising sequences, we can also deduce the expansion for the minimal energy.

\begin{corollary}
    Suppose the sublevel sets $\{ \overline\sigma < \lambda^* g \}$ and $\{ \overline\sigma \leq \lambda^* g \}$ have finite perimeter in $\omega$. Then, the minimum of Problem~\eqref{eq:minFe} is approximated by
    \begin{multline*}
        m_\varepsilon = \lambda^*a + \int_{\{ \overline\sigma < \lambda^* g \}} \overline\sigma(x) - \lambda^*g(x) \diff x + {} \\
        \varepsilon \int_{\partial^*\!E_0 \cap \omega} g(x) \biggl( \int_0^1 \sqrt{1-\sigma(x,t)^2} \diff t \biggr) \dH^{d-1}(x) + o(\varepsilon) ,
    \end{multline*}
    with $E_0 \subset \omega$ any minimiser of~$F^1$.
    \begin{proof}
        Let $E_0 \subset \omega$ be such that $F^1(E_0)$ is well defined, and let $E_\varepsilon$ be minimisers of Problem~\eqref{eq:minFe}. Then there exists a sequence $E_\varepsilon' \to E_0$ in $L^1(\Omega)$ so that
        \[
            F_\varepsilon(E_\varepsilon) \leq F_\varepsilon(E_\varepsilon') = m_0 + \varepsilon F^1(E_0) + o(\varepsilon) ,
        \]
        since $\{ E_\varepsilon \}_{\varepsilon>0}$ is a minimising sequence. It follows that
        \[
            c \Per(E_\varepsilon,\Omega) \leq F^1(E_0) + o(1) , \quad \text{as } \varepsilon \downto 0 ,
        \]
        where $c > 0$ is a lower bound for the weight $g(x) \int_0^1 \sqrt{ 1 - \sigma(x,t)^2 } \diff t$. By the compactness of sets of finite perimeter, for every sequence $\varepsilon_k \downto 0$ we can extract a further subsequence (still indexed by $k$) such that $E_{\varepsilon_k} \to E$ in $L^1(\Omega)$, for some set $E \subset \Omega$ of finite perimeter. The conclusion follows after an application of Theorem~\ref{thm:Gamma-min2}.
    \end{proof}
\end{corollary}

\begin{remark}
    We note that the $O(1)$-scale of the energy derived in Theorem~\ref{thm:Gammalim0} is present only when there is an effective interaction with the solid plates, that is, where $\overline\sigma \neq 0$. If $\overline\sigma \equiv 0$, then the obstacles in equation~\eqref{eq:minF0} collapse to $\varnothing$ and~$\omega$, and we are left only with a weighted isoperimetric problem in~$\omega$. In other words, only the $O(\varepsilon)$~term, of the next order, derived in Theorem~\ref{thm:Gammalim1} remains.

    The case $\sigma \equiv 0$ corresponds to a neutral substrate, that is, a contact angle of~$90^\circ$ (recall that the relative adhesion coefficient is related to Young's contact angle~$\theta$ by $\sigma = -\cos\theta$). Here neither plate energetically prefers to be wetted, so $\overline\sigma \equiv 0$ forces $\lambda^* = 0$ and the zeroth-order energy~$m_0$ vanishes identically; in other words, every admissible region of the prescribed volume is a minimiser of~$F^0$. Consequently, if the relative adhesion is small compared to~$\varepsilon$ in the sense of $\overline\sigma = o(\varepsilon)$, the bulk adhesion term is negligible compared to the perimeter contribution, and thus there is \emph{no filling-up property}. Unlike in Proposition~\ref{prop:E0}, the zeroth-order limit does not localise the liquid according to the gap width, and the wetted region is selected by the first-order weighted-perimeter functional~$F^1$ alone.
\end{remark}

It may well be the case that the set $\{ \overline\sigma \leq \lambda^* g \}$ has infinite perimeter, even for smooth $\overline\sigma$ and~$g$. Even more, it could happen (e.g.~when the minimisers are locked and the limiting set has infinite perimeter) that every minimising family $\{ E_\varepsilon \}_{\varepsilon>0}$ satisfies
\[
    \liminf_{\varepsilon \downto 0} \varepsilon^{-1} ( F_\varepsilon(E_\varepsilon) - m_0 ) = \infty ,
\]
and we cannot guarantee the existence of the next term in the development. However, different scales~$\varepsilon^\beta$, for $\beta \in (0,1]$, may be able to capture the irregularities of the boundary. We exemplify this phenomenon in the following theorem, in which the asymptotic value of the minimal energy~$m_\varepsilon$ in~\eqref{eq:minFe} is bounded in terms of the Minkowski dimension (see Definition~\ref{def:minkowski}) of the boundary of limiting minimising sets.

\begin{theorem}\label{thm:irreg-mins}
    Let $\alpha \in [d-1,d)$. Let $E_0 \subset \omega$ be a set of volume $|E_0| = a$ whose boundary $\partial E_0 \cap \omega$ has Minkowski dimension $\alpha$ and satisfies the density conditions
    \[
        c \leq \frac{ | E_0 \cap B_r(x) | }{ |B_r(x)| } \leq 1 - c , \quad x \in \partial E_0 ,
    \]
    for some $c > 0$ and all sufficiently small $r > 0$. Consider the domain $\Omega_\varepsilon = \omega \times (0,\varepsilon)$, with $\sigma_0 \equiv 0$, and $\sigma_1 = \frac12 \chi_{\omega \setminus E_0}$. Then, the minimum value of the energy~$F_\varepsilon$ has the asymptotic behaviour
    \[
        m_\varepsilon \sim \varepsilon^{d-\alpha} .
    \]
\end{theorem}

\begin{remark}
    A set satisfying the hypotheses of Theorem~\ref{thm:irreg-mins} is, for example, the classical Koch snowflake in the plane, which has a dimension $\alpha = {\log 4} / {\log 3}$. Other dimensions $\alpha \in (1,2)$ can be achieved by modifying the standard construction, replacing at each stage a segment by four other segments scaling at a ratio of $r = 4^{-1/\alpha}$.
\end{remark}

\section{Proofs}\label{sec:proofs}

The proofs of the $\Gamma$-convergence results are split into two standard steps: establishing the lower bound and constructing the recovery sequences (cf.~Definition~\ref{def:Gamma-lim}). Specifically, Theorem~\ref{thm:Gammalim0} follows from Proposition~\ref{prop:F0-liminf} (the $\liminf$~inequality) and Proposition~\ref{prop:F0-limsup} (the existence of recovery sequences). Similarly, the proof of Theorem~\ref{thm:Gammalim1} is a direct consequence of the $\liminf$~inequality in Proposition~\ref{prop:F1-liminf} and the corresponding recovery sequences provided by Proposition~\ref{prop:F1-limsup}.

Before we proceed, it is convenient then to define further
\[
    \Phi^\varepsilon(x,t;\nu,\eta) \coloneq \sqrt{ \varepsilon^2 \abs{ g(x)\nu - \eta\nabla_{\!x} h(x,t) }^2 + \eta^2 } , \quad (x,t) \in \Omega , ~ (\nu,\eta) \in \S^d ,
\]
the anisotropic norm induced by the diffeomorphism~$\phi_\varepsilon$ defined in~\eqref{eq:diffeo}, and the function
\[
    \sigma^\varepsilon(x,t) \coloneq \bigl( (t-1) \sigma_0(x) + t \sigma_1(x) \bigr) \sqrt{ 1 + \varepsilon^2 \abs{ \nabla_{\!x} h(x,t) }^2 } ,
\]
which interpolates the relative adhesion coefficients to the interior of~$\Omega$. A simple integration by parts then yields
\begin{equation}\label{eq:F}
    F_\varepsilon(E) = \int_{\partial^*\!E \cap \Omega} \Phi^\varepsilon(x,t;\nu,\eta) - \sigma^\varepsilon(x,t) \eta \dH^d(x,t) + \int_E \frac{\partial\sigma^\varepsilon}{\partial t}(x,t) \diff{x,t} .
\end{equation}

\begin{proposition}\label{prop:F0-liminf}
    Let $E_\varepsilon \to E$ in $L^1(\Omega)$ be such that
    \[
        \liminf_{\varepsilon \downto 0} F_\varepsilon(E_\varepsilon) < \infty .
    \]
    Then, the characteristic function~$\chi_E$ has a weak derivative~$\mu_E$ in the $t \in (0,1)$ variable of bounded variation in~$\Omega$, and
    \[
        \liminf_{\varepsilon \downto 0} F_\varepsilon(E_\varepsilon) \geq F^0(E) .
    \]
    \begin{proof}
        Let $E_\varepsilon \to E$ in $L^1(\Omega)$. Without loss of generality, suppose that
        \[
            \sup_{\varepsilon > 0} F_\varepsilon(E_\varepsilon) < \infty .
        \]
        From the simple estimate $\Phi^\varepsilon(x,t;\nu,\eta) \geq \abs{ \eta }$, we have
        \[
            F_\varepsilon(E_\varepsilon) \geq \int_{\partial^*\!E_\varepsilon \cap \Omega} \abs{ \eta^\varepsilon } \dH^d - \norm{ \sigma^\varepsilon }_{L^\infty} \int_{\partial^*\!E_\varepsilon \cap \partial\Omega} \dH^d .
        \]
        Also, we trivially have
        \[
            \abs{ \partial\Omega } \geq \int_{\partial^*\!E_\varepsilon \cap \partial\Omega} \dH^d = \int_{\partial^*\!E_\varepsilon \cap \partial\Omega} \abs{ \eta } \dH^d ,
        \]
        since $\abs{ \eta } = 1$ $\Hausdorff^d$-almost everywhere on~$\partial\Omega$. This implies that, up to subsequence, the vertical components of the Gauss--Green measures converge weakly to some Radon measure~$\mu$:
        \[
            \mu^\varepsilon \coloneq \eta^\varepsilon \dH^d |_{\partial^*\!E_\varepsilon} \weakstarto \mu , \quad \text{as } \varepsilon \downto 0 .
        \]
        This measure is characterised by the integration by parts in the vertical direction \dash for every $\varphi \in C^\infty_0(\R^d \times \R)$,
        \[
            \int_E \frac{\partial\varphi}{\partial t} \diff{x,t} = \lim_{\varepsilon \downto 0} \int_{E_\varepsilon} \frac{\partial\varphi}{\partial t} \diff{x,t} = \lim_{\varepsilon \downto 0} \int_{\partial^*\!E_\varepsilon} \eta^\varepsilon \varphi \dH^d = \int_{\R^d \times \R} \varphi \diff\mu .
        \]
        Note also that by our assumptions $\sigma^\varepsilon$ converges uniformly to the function $\sigma(x,t) \coloneq (t-1) \sigma_0(x) + t \sigma_1(x)$, so that
        \begin{align}\label{eq:limse}
            \lim_{\varepsilon \downto 0} \int_{E_\varepsilon} \frac{\partial\sigma^\varepsilon}{\partial t} \diff{x,t} - \int_{\partial^*\!E_\varepsilon \cap \Omega} \sigma^\varepsilon \eta^\varepsilon \dH^d & = \int_E \frac{\partial\sigma}{\partial t}(x,t) \diff{x,t} - \int_\Omega \sigma \diff\mu \notag \\
                                                                                                                                                                                                                & = \int_{\partial\Omega} \sigma \diff\mu .
        \end{align}
        On the other hand, since
        \[
            \int_{\partial^*\!E \cap \Omega} \abs{ \eta } \dH^d(x,t) = \sup_{\substack{\varphi \in C^\infty_0(\Omega) \\ \abs{ \varphi } \leq 1}} \int_E \frac{\partial\varphi}{\partial t} \diff{x,t} ,
        \]
        we can easily deduce that
        \begin{align}\label{eq:liminfPhi}
            \liminf_{\varepsilon \downto 0} \int_{\partial^*\!E_\varepsilon \cap \Omega} \Phi^\varepsilon(x,t;\nu^\varepsilon,\eta^\varepsilon) \dH^d(x,t) & \geq \liminf_{\varepsilon \downto 0} \int_{\partial^*\!E_\varepsilon \cap \Omega} \abs{ \eta^\varepsilon } \dH^d(x,t) \notag \\
                                                                                                                                                           & \geq \abs{ \mu } (\Omega) .
        \end{align}
        Putting~\eqref{eq:limse} and~\eqref{eq:liminfPhi} together yields
        \[
            \liminf_{\varepsilon \downto 0} F_\varepsilon(E_\varepsilon) \geq \abs{ \mu } (\Omega) + \int_{\partial\Omega} \sigma \diff\mu = F^0(E) . \qedhere
        \]
    \end{proof}
\end{proposition}

Before we turn to the construction of recovery sequences for the energy~$F^0$ in Proposition~\ref{prop:F0-limsup}, we show two technical lemmas needed for its proof.

\begin{lemma}\label{lem:mu0}
    Let $E \subset \Omega$ be measurable. Then, $\abs{ \mu_E } (\Omega) = 0$ if, and only if, $E = E_0 \times (0,1)$ for some measurable $E_0 \subset \omega$.
    \begin{proof}
        It is easy to see that $\mu_{E_0 \times (0,1)} = 0$ as, by Fubini's theorem,
        \[
            \int_E \frac{\partial\varphi}{\partial t} \diff{x,t} = \int_{E_0} \int_0^1 \frac{\partial\varphi}{\partial t} \diff t \diff x = 0
        \]
        for every $\varphi \in C^\infty_0(\Omega)$.

        To prove the converse, note that by definition
        \[
            \abs{ \mu_E } (\Omega) = \sup_{\substack{\varphi \in C^\infty_0(\Omega) \\ \abs{ \varphi } \leq 1}} \int_E \frac{\partial\varphi}{\partial t} \diff{x,t} .
        \]
        Hence, $\mu_E = 0$ implies that, actually, \emph{for every} $\varphi \in C^\infty_0(\Omega)$ we have that
        \[
            \int_{\R^d \times \R} \chi_E(x,t) \frac{\partial\varphi}{\partial t} \diff{x,t} = 0 .
        \]
        Thus, the characteristic function $\chi_E \in L^1(\Omega)$ has a weak partial derivative in the vertical direction which is identically zero. Therefore, by Fubini's theorem, for almost every $x \in \omega$ the function $\chi_E(x,\blank)$ is equal to a constant, and since it is a characteristic function its value must be either $0$~or~$1$. Hence, up to a set of measure zero, $E$ coincides with $\{ x \in \omega : \chi_E(x,\blank) \equiv 1 \} \times (0,1)$, as we were to prove.
    \end{proof}
\end{lemma}

The next approximation lemma might be well known, but we were unable to locate a reference in the literature in this specific form. For a similar approach in the particular case $g \equiv 1$ and~$\abs{ \mu_E }$ is replaced by the surface measure~$\Hausdorff^d|_{\partial^*\!E}$, see~\cite{maggi:2012}*{Lemma~17.21}.

\begin{lemma}\label{lem:approx-E}
    Let $E \subset \Omega$ be a set such that the measure~$\mu_E$ is defined and of finite total variation in~$\Omega$. Then, there exists a sequence $\{ E_n \}_{n=1}^\infty$ of smooth sets in~$\R^{d+1}$ converging to~$E$ in $L^1(\Omega)$, such that
    \begin{itemize}
        \item $\displaystyle \int_{E_n} g(x) \diff{x,t} = \int_E g(x) \diff{x,t}$ for all $n \in \N$,
        \item $\displaystyle \lim_{n \to \infty}{ \abs{ \mu_{E_n} }(\Omega) } = \abs{ \mu_E }(\Omega)$.
    \end{itemize}
    \begin{proof}
        We follow a standard regularisation procedure (cf.~\cite{maggi:2012}*{Theorem~13.8}). Let $\{ \varrho_n \}_{n=1}^\infty$ be the standard sequence of mollifiers, and set $f_n \coloneq \chi_E * \varrho_n$. We have $f_n \to \chi_E$ in $L^1(\Omega)$ as $n \to \infty$, with $0 \leq f_n \leq 1$ for all $n \in \N$. By the Morse--Sard lemma, we find that for almost every $s \in (0,1)$, the superlevel sets $E^s_n \coloneq \{ f_n > s \}$ have smooth boundaries. Moreover, $E^s_n \to E$ in $L^1(\Omega)$ as $n \to \infty$, and by the lower semicontinuity of the total variation, the measures
        \[
            \mu_{E^s_n} = \abs{ \nabla\!f_n }^{-1} \frac{\partial\!f_n}{\partial t} \Hausdorff^d \big|_{\partial E^s_n \cap \Omega}
        \]
        satisfy
        \[
            \liminf_{n \to \infty}{ \abs{ \mu_{E^s_n} }(\Omega) } \geq \abs{ \mu_E }(\Omega) .
        \]
        On the other hand, by the coarea formula and Fatou's lemma,
        \begin{multline*}
            \abs{ \mu_E }(\Omega) = \lim_{n \to \infty} \int_\Omega{ \abs*{ \frac{\partial\!f_n}{\partial t} } } \diff{x,t} = \lim_{n \to \infty} \int_0^1 \mkern-9mu \int_{ \partial E^s_n \cap \Omega } \abs{ \nabla\!f_n }^{-1} \abs*{ \frac{\partial\!f_n}{\partial t} } \dH^d \diff s \geq \\
            \int_0^1 \liminf_{n \to \infty}{ \abs{ \mu_{E^s_n} }(\Omega) } \diff s .
        \end{multline*}
        Therefore, we can fix $s \in (0,1)$ and pass to a subsequence (still indexed by~$n$) so that $E'_n \coloneq E^s_n \to E$ and $\abs{ \mu_{E'_n} }(\Omega) \to \abs{ \mu_E }(\Omega)$.

        We now show that the sets $\{ E_n' \}_{n=1}^\infty$ can be deformed to a sequence $\{ E_n \}_{n=1}^\infty$ so that, in addition, the equality
        \begin{equation}\label{eq:int-g}
            \int_{E_n} g(x) \diff{x,t} = \int_E g(x) \diff{x,t}
        \end{equation}
        holds for~$n$ sufficiently large.

        Suppose first that $\mu_E = 0$. Then, by Lemma~\ref{lem:mu0}, $E = E_0 \times (0,1)$ for some subset~$E_0$ of~$\omega$. Moreover, by the construction above, the sets~$E_n'$ are also of the form $E_{0,n}' \times (0,1)$, so $\mu_{E_n'} = 0$ too. Now, if $E_0$~is not the whole~$\omega$ (otherwise the result would be trivial), there exist a smooth vector field~$X$ and some constant $\delta > 0$, such that
        \begin{equation}\label{eq:vol-lower-bound}
            \int_{E_0} \div X \diff x \geq 2\delta .
        \end{equation}
        Suppose for a moment that~$g$ is a smooth function and consider the flow $\{ \psi_s \}_{s \in \R}$, generated by the smooth vector field $Y \coloneq g^{-1} X$. Define the functions
        \[
            v_n(s) \coloneq \int_{\psi_s(E_{0,n}')} g(x) \diff x , \quad s \in \R .
        \]
        Then, by classical formulas we have for all $s \in \R$
        \begin{align}
            v_n'(s)  & = \int_{\psi_s(E_{0,n}')} \div(gY) \diff x = \int_{\psi_s(E_{0,n}')} \div X \diff x , \label{eq:v1} \\
            v_n''(s) & = \int_{\psi_s(E_{0,n}')} \div \bigl( g^{-1}( \div X ) X \bigr) \diff x . \label{eq:v2}
        \end{align}
        By the $L^1(\Omega)$-convergence, it follows from~\eqref{eq:vol-lower-bound} and~\eqref{eq:v1} that $v_n'(0) \geq \delta$ for~$n$ large enough. On the other hand, it follows from~\eqref{eq:v2} that there is a positive constant~$C$ depending only on $\norm{ g }_{W^{1,\infty}}$ and $\norm{ X }_{W^{2,\infty}}$, such that $\abs{ v_n''(s) } \leq C$. Hence, there exists $\tau > 0$ independent of $n$ such that the functions~$v_n$ are strictly increasing over the interval $(-\tau,\tau)$, with the uniform bound
        \[
            \abs[\bigg]{ \int_{\psi_s(E_{0,n}')} g(x) \diff x - \int_{E_0} g(x) \diff x } \geq \delta \abs{ s } - \norm{ g }_{L^\infty} |E_{0,n}' \symdiff E_0|, \quad \text{for all} \quad \abs{ s } < \tau .
        \]
        Therefore, we can select an $s_n \in (-\tau,\tau)$ such that $s_n \to 0$ as $n \to \infty$, and~\eqref{eq:int-g} holds for $E_n \coloneq \psi_{s_n}(E_{0,n}') \times (0,1)$ and $n$ sufficiently large.

        For a general $g \in W^{1,\infty}(\Omega)$, we approximate it with its standard mollifications $g_k \coloneq g * \varrho_k$ ($k \in \N$). Since $g_k \to g$ in $L^\infty(\Omega)$, while $\norm{ \nabla g_k }_{L^\infty} \leq \norm{ \nabla g }_{L^\infty}$ for all $k$, this approximation adds just a small error in the previous computations which is independent of $n \in \N$.

        Now, suppose that $\abs{ \mu_E }(\Omega) \neq 0$. We can find a smooth $\varphi \in C^\infty_0(\Omega)$ such that
        \[
            \int_E \frac{\partial\varphi}{\partial t} \diff{x,t} > \abs{ \mu_E }(\Omega) - \delta .
        \]
        We consider the flow~$\psi_s$ generated by the vector field $\mathbf Y \coloneq (0, \ldots, \varphi/g)$ in $\R^{d+1}$. As $\div( g\mathbf Y ) = \frac{ \partial\varphi }{ \partial t }$, the analogous computations are carried out as before, and we obtain a sequence of sets $\{ E_n \}_{n=1}^\infty$ satisfying the volume constraint~\eqref{eq:int-g}.

        Furthermore, in this case the equality $\abs{ \mu_{E_n} }(\Omega) = \abs{ \mu_{E_n'} }(\Omega)$ also holds for all $n \in \N$. Indeed, $\psi_s(x,t) = (x,\xi_s(x,t))$ for some functions~$\xi_s$, so the Jacobian is given by
        \[
            \det\nabla\psi_s = \frac{ \partial\xi_s }{ \partial t } .
        \]
        Hence, for any given function $\zeta \in C^\infty_0(\Omega)$,
        \[
            \int_{E_n} \frac{ \partial\zeta }{ \partial t } \diff{x,t} = \int_{E_n'} \frac{ \partial\zeta }{ \partial t } \bigl( \xi_s(x,t) \bigr) \frac{ \partial\xi_s }{ \partial t } \diff{x,t} = \int_{E_n'} \frac{ \partial }{ \partial t } (\zeta \circ \xi_s) \diff{x,t} .
        \]
        Taking the supremum over all smooth functions $\abs{ \zeta } \leq 1$ shows the total variations agree.
    \end{proof}
\end{lemma}

\begin{proposition}\label{prop:F0-limsup}
    Let $E \subset \Omega$ be a set such that $F^0(E)$ is finite. Then, there exist a sequence of sets $\{ E_\varepsilon \}_{\varepsilon>0}$ converging to~$E$ in $L^1(\Omega)$ such that
    \[
        \lim_{\varepsilon \downto 0} F_\varepsilon(E_\varepsilon) = F^0(E) .
    \]
    \begin{proof}
        From the Proof of Proposition~\ref{prop:F0-liminf}, it follows that it is sufficient to show that the equality in~\eqref{eq:liminfPhi} can be achieved. If the set~$E$ has finite perimeter, we may take $E_\varepsilon = E$ for all $\varepsilon > 0$, as it is easy to see that
        \[
            \lim_{\varepsilon \downto 0} F_\varepsilon(E) = F^0(E)
        \]
        holds pointwise. When $E$~is~not of finite perimeter, we employ Lemma~\ref{lem:approx-E} to find a sequence $\{ E_n \}_{n=1}^\infty$ of smooth sets converging to~$E$ so that $\lim_{n \to \infty} F^0(E_n) = F^0(E)$. A diagonal argument shows the existence of a recovery sequence for~$E$.
    \end{proof}
\end{proposition}

We prove Proposition~\ref{prop:E0}, which we recall characterises the minimisers of the $\Gamma$-limit~$F^0$.

\begin{proof}[Proof of Proposition~\ref{prop:E0}]
    First, note that
    \[
        F^0(E) \geq (1 - \norm{ \sigma }_{L^\infty}) \abs{ \mu_E } (\Omega) + \int_E \overline\sigma \diff{x,t} \geq \int_E \overline\sigma \diff{x,t} ,
    \]
    with equality if, and only if, $\abs{ \mu_E } (\Omega) = 0$. If we write
    \[
        f(x) \coloneq \int_0^1 \chi_E(x,t) \diff t , \quad x \in \omega ,
    \]
    then also
    \begin{equation}\label{eq:constraints-f}
        0 \leq f \leq 1 , \qquad \int_\omega f(x) g(x) \diff x = a .
    \end{equation}
    Therefore,
    \[
        \min F^0(E) \geq \min \int_\omega f(x) \overline\sigma(x) \diff x ,
    \]
    where the minimisation is done over functions~$f$ subject to the constraints in~\eqref{eq:constraints-f}. It is enough then to show, thanks to Lemma~\ref{lem:mu0}, that functions $f_0 = \chi_{E_0}$ of the proposed form minimise the latter integral.

    This minimisation is a classical problem \dash see, for example, ‘the bathtub principle’ in~\cite{lieb-loss:2001}*{Theorem~1.14}. We provide a proof for the sake of completeness.

    So let $\lambda^* \leq \sup_\omega \overline\sigma/g$ be the supremum over all $\lambda$ such that
    \[
        \int_{ \{ \overline\sigma < \lambda g \} } g(x) \diff x < a .
    \]
    Now, let $f_0 \in L^\infty(\omega)$ be any measurable function such that it satisfies the conditions in~\eqref{eq:constraints-f}, together with $\chi_{\{\overline\sigma<\lambda^*g\}} \leq f_0 \leq \chi_{\{\overline\sigma\leq\lambda^*g\}}$, and let $f$ be another competitor. Let us also denote by $E \coloneq \{ \overline\sigma < \lambda^* g \}$. We then see
    \begin{align}
        \int_\omega f(x) \overline\sigma(x) \diff x & = \int_E f(x) \overline\sigma(x) \diff x + \int_{\omega \setminus E} f(x) \overline\sigma(x) \diff x \notag                  \\
                                                    & \geq \int_E f(x) \overline\sigma(x) \diff x + \lambda^* \int_{\omega \setminus E} f(x) g(x) \diff x \label{ineq:1}           \\
                                                    & = \int_E f(x) \overline\sigma(x) \diff x + \lambda^* \biggl( a - \int_E f(x) g(x) \diff x \biggr) \notag                     \\
                                                    & = \lambda^* \int_\omega f_0(x) g(x) \diff x + \int_E f(x) \bigl( \overline\sigma(x) - \lambda^* g(x) \bigr) \diff x \notag   \\
                                                    & \geq \lambda^* \int_\omega f_0(x) g(x) \diff x + \int_E \overline\sigma(x) - \lambda^* g(x) \diff x \label{ineq:2}           \\
                                                    & = \lambda^* \int_\omega f_0(x) g(x) \diff x + \int_E f_0(x) \bigl( \overline\sigma(x) - \lambda^* g(x) \bigr) \diff x \notag \\
                                                    & = \int_\omega f_0(x) \overline\sigma(x) \diff x . \notag
    \end{align}
    The inequality in~\eqref{ineq:1} is achieved if, and only if, $f = 0$ almost everywhere in $\{ \overline\sigma > \lambda^* g \}$, while the inequality in~\eqref{ineq:2} is achieved if, and only if, $f = 1$ almost everywhere in~$E$.
\end{proof}

\begin{proposition}\label{prop:F1-liminf}
    Let $\{ E_\varepsilon \}_{\varepsilon > 0}$ be a sequence of sets such that $E_\varepsilon \to E$ in $L^1(\Omega)$, and suppose that
    \[
        \liminf_{\varepsilon \downto 0} \varepsilon^{-1} \bigl( F_\varepsilon(E_\varepsilon) - m_0 \bigr) < \infty .
    \]
    Then, $E = E_0 \times (0,1)$, where $E_0 \subset \omega$ is a set of finite perimeter in~$\omega$. Furthermore,
    \[
        \liminf_{\varepsilon \downto 0} \varepsilon^{-1} \bigl( F_\varepsilon(E_\varepsilon) - m_0 \bigr) \geq F^1(E_0) .
    \]
    \begin{proof}
        Let $E_0 \subset \omega$ be a minimiser as given in Proposition~\ref{prop:E0}. We first expand the scaled energy $\varepsilon^{-1} (F_\varepsilon(E) - m_0)$ as
        \begin{multline}\label{eq:energy_expansion}
            F_\varepsilon(E) - m_0 = \int_{\partial^*\!E \cap \Omega} \sqrt{\varepsilon^2 \abs{ g(x)\nu - \eta\nabla_{\!x} h(x,t) }^2 + \eta^2} - \sigma^\varepsilon(x,t) \eta \dH^d(x,t) \\
            + \int_E \frac{\partial\sigma^\varepsilon}{\partial t}(x,t) \diff{x}\diff{t} - \int_{E_0} \overline\sigma(x) \diff x .
        \end{multline}
        Direct computation yields
        \[
            \frac{\partial\sigma^\varepsilon}{\partial t} = \overline\sigma \sqrt{1 + \varepsilon^2 \abs{ \nabla_{\!x} h }^2} + \varepsilon^2 \sigma \frac{\nabla g \cdot \nabla_{\!x} h}{\sqrt{1 + \varepsilon^2 \abs{ \nabla_{\!x} h }^2}} = \overline\sigma + O(\varepsilon^2) .
        \]
        Using this, we can estimate the volume terms in the energy expansion as
        \[
            \int_E \frac{\partial\sigma^\varepsilon}{\partial t} - \int_{E_0} \overline\sigma = \int_E \overline\sigma - \int_{E_0} \overline\sigma + O(\varepsilon^2) \geq -C \varepsilon^2 ,
        \]
        for some $C > 0$.

        To bound the perimeter term in~\eqref{eq:energy_expansion}, we observe that for any $p \in \R^d$ and $\sigma \in (-1,1)$, the strictly convex function $\eta \mapsto \sqrt{ \varepsilon^2 \abs{ p }^2 + \eta^2} - \sigma\eta$ attains its minimum at $\eta = \frac{ \varepsilon \abs{ p } \sigma}{\sqrt{1-\sigma^2}}$, taking the minimum value $\varepsilon \abs{ p } \sqrt{1-\sigma^2}$. Applying this inequality with $p = g \nu - \eta\nabla_{\!x} h$, we obtain
        \begin{align*}
            \sqrt{\varepsilon^2 \abs{ g \nu - \eta\nabla_{\!x} h }^2 + \eta^2} - \sigma^\varepsilon(x,t) \eta & \geq \varepsilon \abs{ g \nu - \eta\nabla_{\!x} h } \sqrt{1 - (\sigma^\varepsilon)^2}                                  \\
                                                                                                              & \geq \varepsilon \bigl( g \abs{ \nu } - \abs{ \nabla_{\!x} h } \abs{ \eta } \bigr) \sqrt{1 - (\sigma^\varepsilon)^2} .
        \end{align*}
        Consequently, for any set $E \subset \Omega$ of finite perimeter, we have the lower bound
        \begin{equation}\label{eq:surface_lower_bound}
            F_\varepsilon(E) - m_0 \geq \varepsilon \int_{\partial^*\!E \cap \Omega} \bigl( g \abs{ \nu } - \abs{ \nabla_{\!x} h } \abs{ \eta } \bigr) \sqrt{1 - (\sigma^\varepsilon)^2} \dH^d + O(\varepsilon^2) .
        \end{equation}

        Now, consider the sequence $E_\varepsilon \to E$ in $L^1(\Omega)$. By passing to a subsequence, we may assume without loss of generality that
        \[
            \sup_{\varepsilon > 0} \varepsilon^{-1} \bigl( F_\varepsilon(E_\varepsilon) - m_0 \bigr) < \infty .
        \]
        Since $\Phi^\varepsilon(x,t;\nu,\eta) \geq \abs{ \eta }$, from~\eqref{eq:energy_expansion} we find there exists a constant $C > 0$, independent of $\varepsilon$, such that
        \begin{equation}\label{eq:eta_bound}
            C\varepsilon \geq (1 - \norm{ \sigma^\varepsilon }_{L^\infty}) \int_{\partial^*\!E_\varepsilon \cap \Omega} \abs{ \eta^\varepsilon } \dH^d .
        \end{equation}
        Substituting~\eqref{eq:eta_bound} into~\eqref{eq:surface_lower_bound}, recalling that $\nabla_{\!x} h$ is bounded, we deduce
        \[
            \varepsilon^{-1} \bigl( F_\varepsilon(E_\varepsilon) - m_0 \bigr) \geq \int_{\partial^*\!E_\varepsilon \cap \Omega} g \abs{ \nu^\varepsilon } \sqrt{1 - (\sigma^\varepsilon)^2} \dH^d + O(\varepsilon) .
        \]

        As $\varepsilon \downto 0$, the term $g \sqrt{1 - (\sigma^\varepsilon)^2}$ converges uniformly to $g \sqrt{1 - \sigma^2}$, which is strictly bounded away from zero. This uniform convergence, together with the uniform energy bound assumed above, implies that the perimeters $\Per(E_\varepsilon,\Omega)$ are uniformly bounded by a constant independent of~$\varepsilon$. Therefore, by the compactness of sets of finite perimeter, the limit~$E$ must be a set of finite perimeter, for which lower semicontinuity yields
        \begin{equation}\label{eq:liminf1}
            \liminf_{\varepsilon \downto 0} \varepsilon^{-1} \bigl( F_\varepsilon(E_\varepsilon) - m_0 \bigr) \geq \int_{\partial^*\!E \cap \Omega} g \abs{ \nu } \sqrt{1 - \sigma^2} \dH^d .
        \end{equation}

        Finally, since $E$ takes the cylindrical form $E = E_0 \times (0,1)$, we have $\abs{ \nu } = 1$ for $\Hausdorff^{d-1}$-almost every point on the reduced boundary by the slicing properties of sets of finite perimeter \dash see, for instance,~\cite{maggi:2012}*{Theorem~18.11}. A direct application of Fubini's theorem in~\eqref{eq:liminf1} then establishes the desired $\liminf$~inequality.
    \end{proof}
\end{proposition}

As before, to aid the construction of recovery sequences for~$F^1$, we need a technical lemma to adjust the volume constraint along the sequence.

\begin{lemma}\label{lem:b0}
    Let $E_0 \subset \omega$ be a set of finite perimeter. Suppose $X \in C^\infty_0(\omega;\R^d)$ is such that
    \[
        \int_{E_0} \div(gX) \diff x = \int_{\partial^*\!E_0} g(x) X \cdot \nu \dH^{d-1}(x) \neq 0 ,
    \]
    and let $\mathbf X(x) \coloneq (X(x),0)$ denote the extension of~$X$ to~$\R^{d+1}$. Let $\mathbf Y \in C^\infty_0(\Omega;\R^{d+1})$ be a vector field, and denote by $\{ \psi_s' \}_{s \in \R}$ the flow it generates. Then, there exists an $\varepsilon_0 > 0$ and a $C^1$~function $b \colon (-\varepsilon_0,\varepsilon_0) \to \R$ such that $b(0) = 0$ and the family of mappings
    \[
        \psi_\varepsilon(x,t) \coloneq \psi_\varepsilon'(x,t) + b(\varepsilon) \mathbf X(x)
    \]
    satisfies
    \[
        \int_{\psi_\varepsilon(E \times (0,1))} g(x) \diff{x,t} = \int_E g(x) \diff x , \quad \text{for all} \quad \abs{ \varepsilon } < \varepsilon_0 .
    \]
    Moreover,
    \begin{equation}\label{eq:b'}
        b'(0) \int_E \div(gX) \diff x = - \int_{E \times (0,1)} \div \bigl( g(x) \mathbf Y(x,t) \bigr) \diff{x,t} .
    \end{equation}
    \begin{proof}
        For $\varepsilon > 0$ and $b \in \R$, define the mappings $\psi_{\varepsilon,b} \coloneq \psi_\varepsilon' + b \mathbf X$, and let
        \[
            V(\varepsilon,b) \coloneq \int_{\psi_{\varepsilon,b}(E \times (0,1))} g(x) \diff{x,t} .
        \]
        Then, $V$ defines a $C^1$~function such that
        \[
            \frac{ \partial V }{ \partial b }(0,0) = \int_{E \times (0,1)} \div(g\mathbf X) \diff{x,t} = \int_E \div(gX) \diff x \neq 0
        \]
        by assumption. The result follows from the implicit function theorem.
    \end{proof}
\end{lemma}

\begin{proposition}\label{prop:F1-limsup}
    For every $E_0 \subset \omega$ in the domain of~$F^1$, there exists a sequence $\{ E_\varepsilon \}_{\varepsilon>0}$ such that
    \[
        \lim_{\varepsilon \downto 0} \varepsilon^{-1} \bigl( F_\varepsilon(E_\varepsilon) - m_0 \bigr) = F^1(E_0) .
    \]
    \begin{proof}
        Let $\delta > 0$ and $X \in C^\infty_0(\omega;\R^d)$ be such that $\norm{ X }_{L^\infty(\omega)} \leq 1$ and
        \[
            \int_{E_0} \div X \diff x = \int_{\partial^*\!E_0} X \cdot \nu \dH^{d-1} > \Per(E_0,\omega) - \delta .
        \]
        In particular, we have
        \[
            \int_{\partial^*\!E_0} g(x) X \cdot \nu \dH^{d-1}(x) \geq \int_{\partial^*\!E_0} g(x) \dH^{d-1} - \delta \norm{ g }_{L^\infty(\omega)} > 0
        \]
        if $\delta$ is sufficiently small. Let $u \in C^\infty(\overline\Omega)$, and set
        \[
            w(x,t) \coloneq c + \int_0^t u(x,s) \diff s ,
        \]
        where $c \in \R$ is a constant to be determined later. Consider the smooth vector field
        \[
            \mathbf Y(x,t) \coloneq w(x,t) \bigl( X(x),0 \bigr) ,
        \]
        and the flow $\{ \psi_s' \}_{s \in \R}$ it generates. According to Lemma~\ref{lem:b0}, we can find a function $b(\varepsilon) = \varepsilon b'(0) + o(\varepsilon)$ such that the family of diffeomorphisms $\psi_\varepsilon \coloneq \psi_\varepsilon' + b(\varepsilon) \mathbf X$ preserves the integral
        \[
            \int_{\psi_\varepsilon(E)} g(x) \diff{x,t} = a ,
        \]
        where, as before, we are denoting $E = E_0 \times (0,1)$. Moreover, thanks to equation~$\eqref{eq:b'}$, we may adjust~$c$ so that $b'(0) = 0$. In this way, $\partial_\varepsilon|_{\varepsilon=0} \psi_\varepsilon = \mathbf Y$. Set $E_\varepsilon \coloneq \psi_\varepsilon(E)$, and let $(\nu^\varepsilon,\eta^\varepsilon) \coloneq \operatorname{cof} \nabla\psi_\varepsilon \cdot (\nu,0)$. If we recall the expansion
        \[
            \operatorname{cof} \nabla\psi_\varepsilon(x,t) = {\id} + \varepsilon \bigl( (\div \mathbf Y){\id} - (\nabla \mathbf Y)^\top \bigr) + o(\varepsilon) ,
        \]
        then it is easy to see that
        \begin{equation}\label{eq:ne-expansion}
            \nu^\varepsilon = \nu + O(\varepsilon) , \qquad \eta^\varepsilon = - \varepsilon u X \cdot \nu + o(\varepsilon) .
        \end{equation}
        On the other hand, by the Lipschitz continuity,
        \[
            g( \psi_\varepsilon(x,t) ) = g(x) + O(\varepsilon) , \qquad \nabla_{\!x} h( \psi_\varepsilon(x,t) ) = \nabla_{\!x} h(x,t) + O(\varepsilon) ,
        \]
        uniformly on $(x,t) \in \Omega$. Thus, the expansion of $\Phi^\varepsilon(\psi_\varepsilon(x,t);\nu^\varepsilon,\eta^\varepsilon)$ up to an $o(\varepsilon)$~error term is given by
        \begin{equation}\label{eq:Phie-expansion}
            \sqrt{ \varepsilon^2 \abs{ g( \psi_\varepsilon(x,t) )\nu^\varepsilon - \eta^\varepsilon \nabla_{\!x} h( \psi_\varepsilon(x,t) ) }^2 + (\eta^\varepsilon)^2 } = \varepsilon \sqrt{ g^2 + u^2(X \cdot \nu)^2 } + o(\varepsilon) .
        \end{equation}
        Finally, the area term on the plates expands as
        \begin{equation}\label{eq:area-expansion}
            \int_{E_\varepsilon} \frac{ \partial\sigma^\varepsilon }{ \partial t } \diff{x,t} = \int_{E_0} \overline\sigma(x) \diff x + \varepsilon \frac{ \diff }{ \diff\varepsilon }\bigg|_{\varepsilon=0} \int_{\psi_\varepsilon(E)} \overline\sigma(x) \diff{x,t} + o(\varepsilon) .
        \end{equation}
        Note that it follows from Proposition~\ref{prop:E0} that $\overline\sigma = \lambda^* g$ on $\partial^*\!E_0 \cap \omega$, and therefore
        \[
            \frac{ \diff }{ \diff\varepsilon }\bigg|_{\varepsilon=0} \int_{\psi_\varepsilon(E)} \overline\sigma(x) \diff{x,t} = \frac{ \diff }{ \diff \varepsilon } \bigg|_{\varepsilon=0} \lambda^* \int_{\psi_\varepsilon(E)} g(x) \diff{x,t} = 0 .
        \]
        Putting equations~\eqref{eq:ne-expansion}, \eqref{eq:Phie-expansion}, and~\eqref{eq:area-expansion} together, we deduce
        \begin{multline*}
            \lim_{\varepsilon \downto 0} \frac{ F_\varepsilon(E_\varepsilon) - m_0 }{ \varepsilon } = \int_{\partial^*\!E_0 \cap \omega} \int_0^1 \sqrt{ g(x)^2 + u(x,t)^2 (X \cdot \nu)^2 } \diff t \dH^{d-1}(x) + {} \\
            \int_{\partial^*\!E_0} \biggl( \int_0^1 u(x,t) \sigma(x,t) \diff t \biggr) X \cdot \nu \dH^{d-1}(x) .
        \end{multline*}
        By the arbitrariness of $u$, we may choose a smooth function which approximates uniformly on compact subsets of $\omega$, up to an error $O(\delta)$, the optimal profile
        \[
            - g(x) \frac{ \sigma(x,t) }{ \sqrt{ 1 - \sigma(x,t)^2 } } .
        \]
        Then, since $X$ approximates $\nu$ in $L^2(\partial^*\!E_0, \Hausdorff^{d-1})$, as
        \[
            \int_{\partial^*\!E_0 \cap \omega}{ \abs{ X - \nu }^2 }\dH^{d-1} \leq 2 \int_{\partial^*\!E_0 \cap \omega} (1-X\cdot\nu) \dH^{d-1} < 2\delta ,
        \]
        we may deduce
        \[
            \lim_{\varepsilon \downto 0} \frac{ F_\varepsilon(E_\varepsilon) - m_0 }{ \varepsilon } = F^1(E_0) + O(\delta) ,
        \]
        for any $\delta > 0$. A standard diagonalisation argument concludes the proof.
    \end{proof}
\end{proposition}

We conclude this section with the proof of Theorem~\ref{thm:irreg-mins}, which establishes the existence of different possible asymptotic behaviours for the minimum of~$F_\varepsilon$, depending on the regularity of the level sets of~$\overline\sigma$.

\begin{proof}[Proof of Theorem~\ref{thm:irreg-mins}]
    We first show the lower bound. Let $E \subset \Omega$ be a set of finite perimeter, and consider its \emph{density function}
    \[
        u(x) \coloneq \int_0^1 \chi_E(x,t) \diff t , \quad x \in \omega .
    \]
    Consider then the hypograph $E^* \coloneq \{ (x,t) \in \Omega : u(x) \geq t \}$. By Fubini's theorem, $|E^*| = |E|$, and by Steiner's inequality (cf.~\cite{maggi:2012}*{Theorem~14.4}), $\Per(E^*,\Omega) \leq \Per(E,\Omega)$. Furthermore, the $t=1$ trace of~$E^*$ is contained in that of~$E$. Therefore, if $E_\varepsilon$~is a minimiser of~$F_\varepsilon$, we may assume without loss of generality that $E_\varepsilon$~is equal (up to sets of measure zero) to the hypograph of its density function $u_\varepsilon \in BV(\omega)$.

    For this type of set, we may compute
    \begin{equation}\label{eq:Fu}
        F_\varepsilon(E_\varepsilon) = \int_{\{0<u_\varepsilon<1\}} \sqrt{ 1 + \varepsilon^2 \abs{ \nabla u_\varepsilon }^2 } \diff x + \int_{\{u_\varepsilon=1\}} \sigma_1(x) \diff x .
    \end{equation}
    Since $|E_\varepsilon| = |E_0|$, we have
    \[
        \int_{E_0} u_\varepsilon \diff x + \int_{\omega \setminus E_0} u_\varepsilon \diff x = \int_{E_0} (1-u_\varepsilon) \diff x ,
    \]
    and so
    \[
        \norm{ u_\varepsilon - \chi_{E_0} }_{L^1(\omega)} = \int_{\omega \setminus E_0} u_\varepsilon \diff x + \int_{E_0} (1-u_\varepsilon) \diff x = 2 \int_{\omega \setminus E_0} u_\varepsilon \diff x .
    \]
    If we recall $\sigma_1 = \frac12 \chi_{\omega \setminus E_0}$, from~\eqref{eq:Fu} we deduce
    \begin{equation}\label{ineq:Fu-bound}
        F_\varepsilon(E_\varepsilon) \geq \varepsilon \abs{ \nabla u_\varepsilon }(\omega) + \frac14 \norm{ u_\varepsilon - \chi_{E_0} }_{L^1(\omega)} .
    \end{equation}

    Now, we localise the last inequality to a small ball $B \coloneq B_\varepsilon(x)$ centred at a point $x \in \partial E_0$. For this, let $\mean{ v }$ denote the average of a function~$v$ over~$B$. By the triangle and Poincaré inequalities,
    \begin{align*}
        \norm{ \chi_{E_0} - \mean{ \chi_{E_0} } }_{L^1(B)} & \leq \norm{ \chi_{E_0} - u_\varepsilon }_{L^1(B)} + \norm{ u_\varepsilon - \mean{ u_\varepsilon } }_{L^1(B)} + \norm{ \mean{ u_\varepsilon } - \mean{ \chi_{E_0} } }_{L^1(B)} \\
                                                           & \leq 2 \norm{ u_\varepsilon - \chi_{E_0} }_{L^1(B)} + C\varepsilon \abs{ \nabla u_\varepsilon }(B) ,
    \end{align*}
    for some $C > 0$ independent of~$\varepsilon$. But by the density assumption on~$\partial E_0$,
    \[
        \norm{ \chi_{E_0} - \mean{ \chi_{E_0} } }_{L^1(B)} = 2 \frac{ |B \cap E_0| |B \setminus E_0| }{ |B| } \geq c \varepsilon^d
    \]
    for some constant $c > 0$ independent of $x$ and~$\varepsilon$. Hence, since the Minkowski dimension of~$\partial E_0$ is~$\alpha$, we can consider $N \sim \varepsilon^{-\alpha}$ such disjoint balls (see Definition~\ref{def:minkowski}), and then it follows from~\eqref{ineq:Fu-bound} that
    \[
        F_\varepsilon(E_\varepsilon) \geq cN \varepsilon^d \sim \varepsilon^{d-\alpha} ,
    \]
    as we were to prove.

    We now show the upper bound. Consider the standard mollifications $u_\varepsilon \coloneq \chi_{E_0} * \varrho_\varepsilon$, and let $E_\varepsilon \coloneq \{ (x,t) \in \omega \times (0,1) : u_\varepsilon(x) \geq t \}$ be the hypograph of $u_\varepsilon$. Then, by the properties of mollifications, the volume $|E_\varepsilon| = a$ remains constant, so we may use formula~\eqref{eq:Fu}.

    Now, since $\{ u_\varepsilon = 1 \}$ is strictly contained in~$E_0$, the second integral in~\eqref{eq:Fu} vanishes. On the other hand, we have the uniform bound
    \[
        \abs{ \nabla u_\varepsilon(x) } \leq \int_{E_0}{ \abs{ \nabla\varrho_\varepsilon(x-y) } }\diff y \leq \frac{ C }{ \varepsilon } , \quad x \in \omega ,
    \]
    while $\nabla u_\varepsilon$ is supported on the $\varepsilon$-neighbourhood $B_\varepsilon(\partial E_0)$. Therefore, by Definition~\ref{def:minkowski} of Minkowski dimension,
    \[
        F_\varepsilon(E_\varepsilon) \leq C |B_\varepsilon(\partial E_0)| \sim \varepsilon^{d-\alpha} . \qedhere
    \]
\end{proof}

\section{Phase-field approximation and numerical results}\label{sec:numerics}

In this section, we construct a phase-field approximation of the dimensionally reduced functional~\eqref{eq:Gamma-expansion} derived above, validate it against the level-set filling-up rule of Proposition~\ref{prop:E0}, and use it to compute capillary force curves under normal motion of the plates.

\subsection{Phase-field model}

The liquid region is described by a smooth function $u \colon \omega \subset \R^d \to \R$ that distinguishes the two phases: $u(x) = 0$ indicates the vapour phase, $u(x) = 1$ the liquid phase, and $u(x) \in (0,1)$ their diffuse interface. The wetted set is recovered \emph{a posteriori} as the level set
\[
    E_0 \coloneq \bigl\{ x \in \omega : u(x) > \tfrac12 \bigr\} .
\]

Throughout this section we assume, for simplicity, that both plates consist of the same homogeneous material. In terms of the surface tensions appearing in Gauss's energy~\eqref{eq:explicit-energy}, the relative adhesion coefficients of the two plates introduced in Section~\ref{sec:setting} then coincide and are spatially constant,
\[
    \sigma_0 = \sigma_1 \equiv \sigma \coloneq \frac{ \gamma_\mathrm{ls} - \gamma_\mathrm{sv} }{ \gamma_\mathrm{lv} } ,
\]
with $-1 < \sigma < 1$ by the wetting condition. Thus, the combined adhesion coefficient is $\overline\sigma = \sigma_0 + \sigma_1 = 2\sigma$, the interpolation of Theorem~\ref{thm:Gammalim1} reduces to $\sigma(x,t) = (2t-1)\sigma$, and the dimensionally reduced functional of Corollary~\ref{cor:Ge} takes the form
\[
    G_\varepsilon(E) = \int_E 2\sigma \diff x + \varepsilon \kappa_\sigma \int_{\partial^*\!E\cap \omega} g(x) \dH^{d-1}(x) ,
\]
with $\kappa_\sigma \coloneq \int_0^1 \sqrt{1-(2t-1)^2\sigma^2} \diff t = \frac12 \bigl( \sqrt{1-\sigma^2} + \frac{\arcsin\sigma}{\sigma} \bigr)$, extended by continuity to~$\kappa_0 = 1$ and decreasing to~$\pi/4$ as~$\abs{\sigma} \to 1$. Following Modica--Mortola~\cites{modica-mortola:1977, modica:1987a}, we approximate $G_\varepsilon$ by its phase-field version
\begin{equation}\label{eq:phase-field-approx-energy}
    I_{\delta,\varepsilon}(u) = 2\sigma \int_\omega u(x) \diff x + \varepsilon \kappa_\sigma \int_\omega g(x) p_\delta(u) \diff x ,
\end{equation}
where
\[
    p_\delta(u) \coloneq \frac{1}{c_0} \bigl( \delta \abs{ \nabla u }^2 + \delta^{-1}
    W(u) \bigr)
\]
is the Modica--Mortola Lagrangian with interface width $\delta > 0$, $W(u) = u^2 (u - 1)^2$ is a double-well penalty, and $c_0 = 2 \int_0^1 \sqrt{W(u)} \diff u$ is a normalising prefactor. The constraint on the total volume now reads
\begin{equation}\label{eq:phase-field-approx-volume}
    V(u) \coloneq \int_\omega u(x)g(x) \diff x = a .
\end{equation}

The bulk term~$2\sigma u$ of~$I_{\delta,\varepsilon}$ reproduces the wetting energy $\int_E \overline\sigma$ of the dimensionally reduced functional~$G_\varepsilon$ derived in Corollary~\ref{cor:Ge}, and the diffuse-interface term~$p_\delta(u)$ approximates the weighted perimeter. Following~\cite{owen-sternberg:1991}, $I_{\delta,\varepsilon} \stackrel{\Gamma}{\to} G_\varepsilon$ in the $L^1$ topology as~$\delta \to 0$, so minimisers of~$I_{\delta,\varepsilon}$ approximate minimisers of the reduced problem.

The $\Gamma$-equivalence of $G_\varepsilon$ and~$F_\varepsilon$ guarantees only that the approximation of the \emph{energy} improves to order~$o(\varepsilon)$; no result guarantees that minimisers of~$G_\varepsilon$ are closer to minimisers of~$F_\varepsilon$ than minimisers of~$F^0$ are. Since the interface regularisation term~$\varepsilon F^1$ is present in~$F_\varepsilon$, we nevertheless expect that including it yields a better approximation of the minimisers than computing with~$F^0$ alone.

\subsection{Numerical setup}\label{subsec:numerical-setup}

\begin{figure}
    \includegraphics[width=1\linewidth]{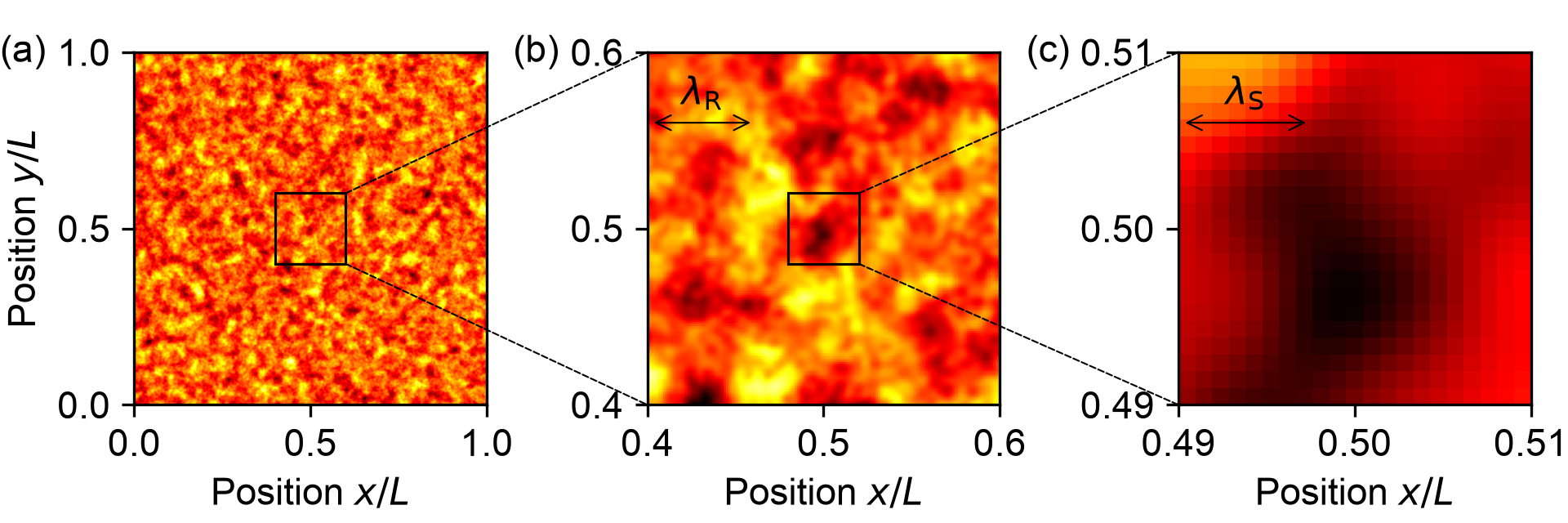}
    \caption{\textbf{Gap profile.} The plot shows the gap formed by two parallel rough surfaces, computer-generated by Fourier filtering as described in Subsection~\ref{subsec:numerical-setup}. The most prominent peaks and valleys have a length scale comparable to the roll-off wavelength $\lambda_\mathrm{R}$, while the gap is smooth below the short cut-off wavelength $\lambda_\mathrm{S}$; the two lengths are marked in panels (b) and (c), respectively. The black squares in panels (a) and (b) indicate the areas shown in panels (b) and (c), respectively.}\label{fig:roughness-setup}
\end{figure}

We discretise both the phase field~$u$ and the plate profiles $h_0,h_1$ on a regular, periodic two-dimensional grid~$\omega$ of $N_x = N_y = 1024$ points with P1 (piecewise linear) finite elements. The lateral domain size is~$L=1$, so the grid spacing is $l = L / N_x$.

We generate the two rough surfaces $h_0,h_1 \colon \omega \to \R$ (lower and upper plate) following~\cite{jacobs2017quantitative}. We start from an isotropic power spectral density~$C^\mathrm{iso}$ of the height (the squared amplitude of its Fourier modes as a function of the wavevector magnitude~$q$), which is constant for wavelengths between the long cut-off $\lambda_\mathrm{L} = L$ and the roll-off $\lambda_\mathrm{R} = 64l$, decays below $\lambda_\mathrm{R}$ as the power law $C_0 q^{-2-2H}$ with Hurst exponent $H=0.8$ and prefactor $C_0=10L^{2-2H}$, and vanishes below the short cut-off wavelength $\lambda_\mathrm{S}=8l$. Imposing random phase angles drawn from the uniform distribution, we perform an inverse Fourier transform to obtain the height profiles in real space. The mean of the surface height profiles $h_0$ and~$h_1$ is zero, the standard deviation is approximately~$0.03L$, and the maximal deviation is approximately~$0.21L$.

Following the setting of Section~\ref{sec:setting}, we fix the plate-separation parameter $\varepsilon = 10^{-3}$: the plates are the graphs of $\varepsilon h_0$ and~$\varepsilon z + \varepsilon h_1$, so the physical gap is~$\varepsilon g(x)$, with rescaled gap $g(x) = z + h_1(x) - h_0(x)$ and mean separation $z \in [0.5L, 1.5L]$ (see Figure~\ref{fig:roughness-setup} for an illustration of the rescaled gap at maximal mean separation). The physical liquid volume is correspondingly~$\varepsilon a$. We set the liquid volume so that $a = \frac14 L^3$ and the interface width to $\delta = 2l$. By the Young--Dupré relation, the constant adhesion coefficient corresponds to a spatially uniform contact angle~$\theta$ through $\sigma = -\cos\theta$. Below we report results for contact angles between $\theta = 5^\circ$ and~$\theta = 175^\circ$.

We run quasi-static simulations in which one plate is held fixed while the other is displaced normally to mimic approach and retraction: at each separation the phase field from the converged solution of the previous step relaxes to a minimiser before the plates move again. Since the volume constraint~\eqref{eq:phase-field-approx-volume} is linear in~$u$, we enforce it by projection: a projected L-BFGS method~\cite{Nocedal2006-gn} keeps every iterate on the plane~$V(u) = a$ and restricts the search directions to its tangent plane, replacing the gradient by $\nabla I_{\delta,\varepsilon} - \lambda \nabla V$ with the closed-form multiplier $\lambda = \scalar{ \nabla V }{ \nabla I_{\delta,\varepsilon} } / \norm{\nabla V}^2$. At convergence the projected gradient vanishes, so the minimiser satisfies the Euler--Lagrange equation $\delta I_{\delta,\varepsilon}/\delta u = \lambda g$ with the Lagrange multiplier~$\lambda$ of the volume constraint.

The separation follows a prescribed sequence~$\{z_i\}_i$ satisfying $0.5L \le z_i \le 1.5L$ with step size $\abs{ z_{i+1} - z_i } = 0.01L$. The sequence starts at the minimal separation $z_1 = 0.5L$. We initialise the first step with a phase field in which the liquid is condensed into a central square of side length $\frac{\sqrt 2}{2} L$, with vapour elsewhere. At the initial separation this square matches the prescribed liquid volume, as~$\bigl(\tfrac{\sqrt2}{2}\bigr)^2 L^2 \cdot \tfrac12 L = \tfrac14 L^3 = a$.

\subsection{Morphology of the wetted region}

\begin{figure}
    \includegraphics[width=1\linewidth]{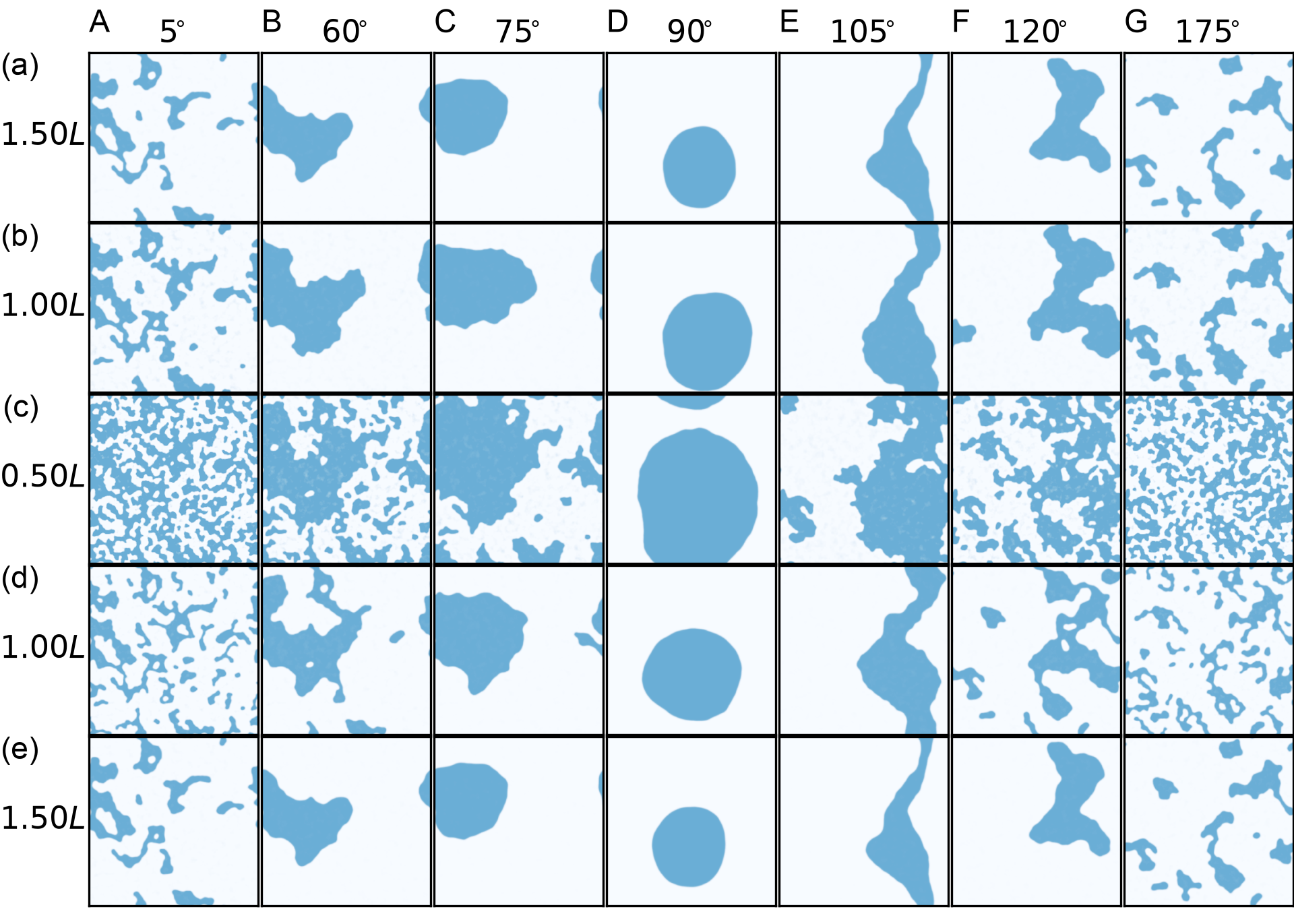}
    \caption{\textbf{Last round trip.} The figure shows the evolution of minimisers of the phase field obtained in the last round trip. Panels (a) to~(c) correspond to the approach of the moving plate, and panels (c) to~(e) to the retraction. The mean separations are $z = 1.5L$ in (a) and~(e), $z = 1.0L$ in (b) and~(d), and $z = 0.5L$ in~(c). Because the volume is kept constant, the ratio of the liquid volume to the total available volume is $16.7\%$ in (a) and~(e), $25\%$ in (b) and~(d), and $50\%$ in~(c). The columns correspond to different contact angles, $\theta = 5^\circ$ in column~A, $60^\circ$ in~B, $75^\circ$ in~C, $90^\circ$ in~D, $105^\circ$ in~E, $120^\circ$ in~F and $175^\circ$ in~G. At the same mean separation and contact angle, the phase field differs between approach and retraction (compare rows (b) and~(d)): the configuration is hysteretic.}\label{fig:comparison}
\end{figure}

After an initial retraction from $z = 0.5L$ to the maximal separation $z = 1.5L$, the moving (upper) plate approaches the fixed (lower) plate over $100$~discrete steps until the minimal separation $z = 0.5L$, and then retracts; this round trip is repeated $10$~times to remove the influence of the initial configuration. Figure~\ref{fig:comparison} shows the converged phase field of the last round trip for a range of mean separations and contact angles.

For the nearly perfectly hydrophilic ($\theta = 5^\circ$) and nearly perfectly hydrophobic ($\theta = 175^\circ$) cases the morphology is governed by the surface topography and the liquid breaks into many small droplets, whereas at the neutral angle ($\theta = 90^\circ$) it forms a single large droplet. This is the phase-field signature of the analysis of Section~\ref{sec:setting}: the energy~\eqref{eq:phase-field-approx-energy} splits into a bulk wetting term, weighted by~$-\cos\theta$, and a perimeter term. At the extreme contact angles (columns A and~G in Figure~\ref{fig:comparison}) the wetting term dominates and the liquid follows the gap, recovering the level-set filling of Proposition~\ref{prop:E0}; at the neutral angle the wetting weight vanishes and the configuration is governed entirely by the perimeter term~$F^1$ of Theorem~\ref{thm:Gammalim1}, so the liquid coalesces into a single droplet (column~D).

We examine the extreme contact angles more closely in Figure~\ref{fig:affinity-difference}, overlaying the phase field on the level set. At the smallest separation, panel~(a) of Figure~\ref{fig:affinity-difference}, the simulated minimiser agrees closely with the islands obtained from the filling-up rule of Proposition~\ref{prop:E0}, capturing all but the smallest islands of the level-set. The residual discrepancy is the first-order effect predicted by Theorem~\ref{thm:Gammalim1}: the perimeter term smooths the rough island boundaries, rounding the droplets into more spherical shapes and merging nearby islands.

\begin{figure}
    \includegraphics[width=1\linewidth]{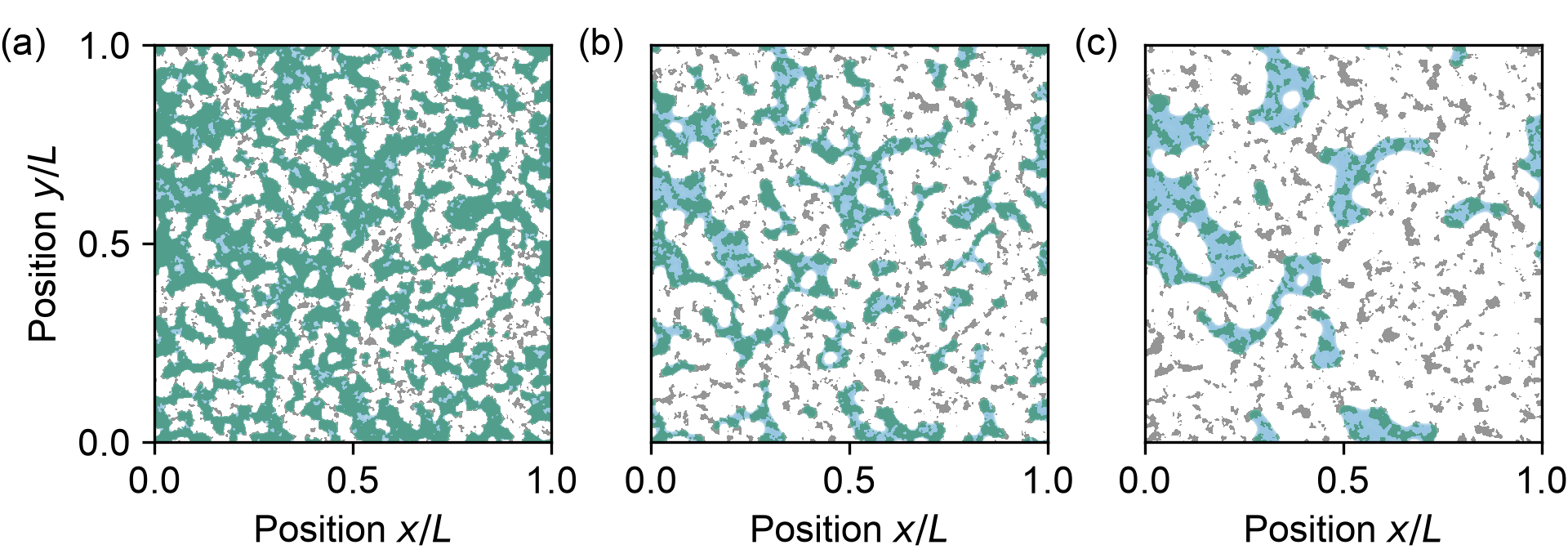}
    \caption{\textbf{Simulation vs.~level-set capillary condensation.} The figure shows the distribution of the condensed liquid (the phase field) for nearly perfectly hydrophilic cases (contact angle $\theta=5^\circ$) at different mean separations, $z = 0.5L$ in~(a), $z = 1.0L$ in~(b) and $z = 1.5L$ in~(c). Superimposed on the minimiser of the phase field (blue) is the level set (grey). Overlapping regions are emphasised (green). The overlapping regions are larger at smaller mean separations.}\label{fig:affinity-difference}
\end{figure}

At larger separations, panels (b) and~(c), the minimiser deviates from the level set. It still covers islands identified by the level set, but only selectively, with bridges of liquid between these islands. At the largest separation, panel~(c), even the bridges of liquid are smoothed to be more spherical, indicating that the perimeter contribution~$F^1$ becomes more prominent relative to the bulk wetting term at larger separation.

\subsection{Capillary forces}

\begin{figure}
    \includegraphics[width=1\linewidth]{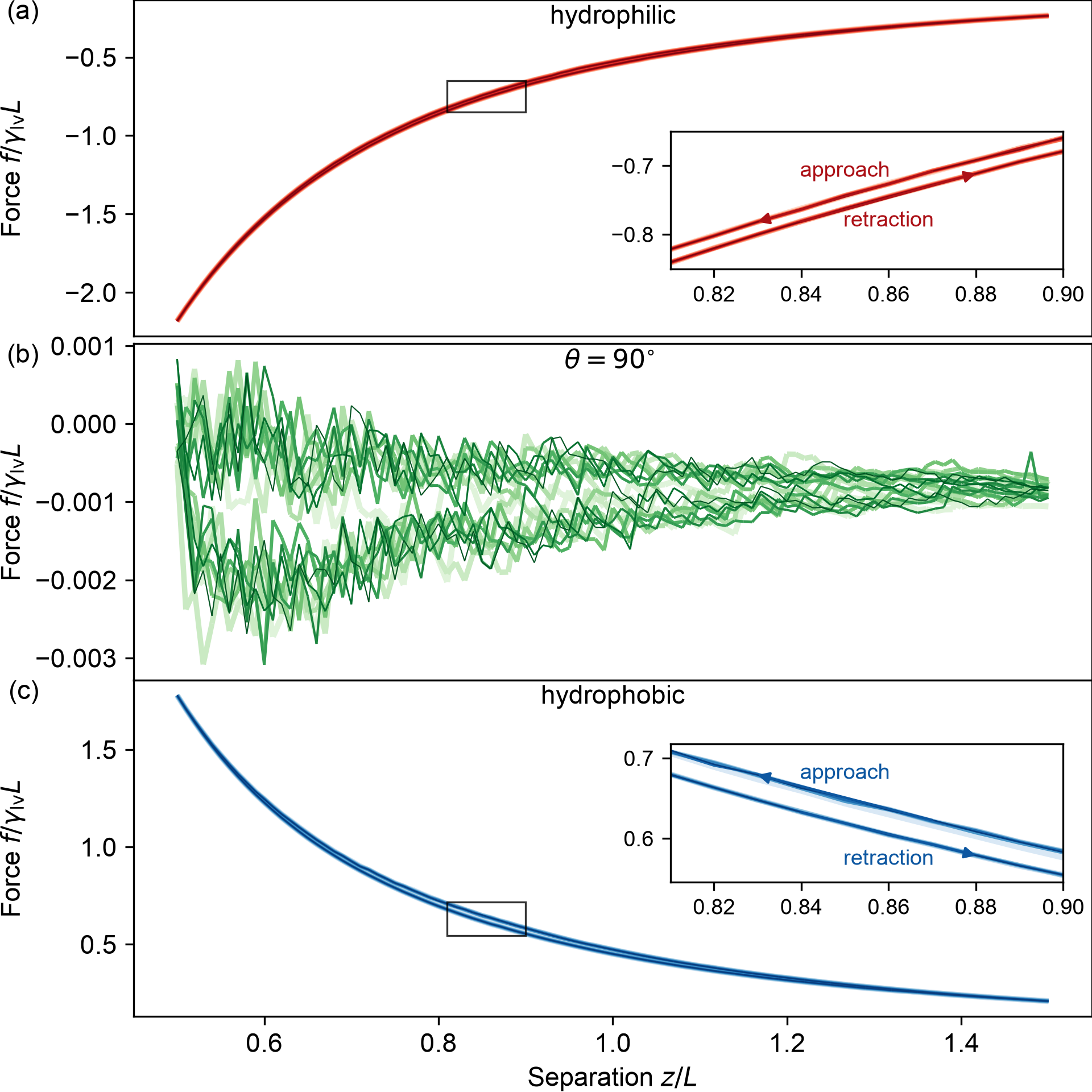}
    \caption{\textbf{Capillary force.} The plot shows the capillary forces computed from the phase field at different separations. The simulations follow the quasi-static approach--retraction protocol of Subsection~\ref{subsec:numerical-setup}. Each curve shows $10$~approach and retraction cycles. Panel~(a) shows data for hydrophilic interfaces ($\theta = 5^\circ$), panel~(b) shows $\theta = 90^\circ$ where the bulk wetting term disappears, and~(c) shows hydrophobic interfaces ($\theta = 175^\circ$). The insets in panels (a) and~(c) highlight the hysteresis.}\label{fig:capillary-force}
\end{figure}

Finally, we compute the macroscopic capillary force at each separation (Figure~\ref{fig:capillary-force}) as
\begin{equation}\label{eq:phase-field-force}
    f = \int_\omega \lambda u - \varepsilon \kappa_\sigma p_\delta(u) \diff x .
\end{equation}
Here $\lambda$ denotes the Lagrange multiplier associated with the volume constraint~\eqref{eq:phase-field-approx-volume} at the given separation, obtained in closed form from the converged solution as described in Subsection~\ref{subsec:numerical-setup}. It plays the role that the threshold~$\lambda^*$ of Proposition~\ref{prop:E0} plays for the limit functional~$F^0$. The expression~\eqref{eq:phase-field-force} is the phase-field counterpart of the classical sharp-interface capillary force $\lambda \abs{E_0} - \varepsilon \kappa_\sigma \Per(E_0, \omega)$, which arises by formally differentiating the reduced energy~$G_\varepsilon$ with respect to the mean separation at fixed volume. Its two terms are the pressure force acting on the wetted region and the line-tension force acting on its boundary. This derivation is formal: the $\Gamma$-convergence results of Section~\ref{sec:setting} control the energies themselves, not their derivatives with respect to the separation. Forces are reported in units of the liquid-vapour surface tension~$\gamma_\mathrm{lv}$, see equation~\eqref{eq:explicit-energy}.

For hydrophilic interfaces the force is attractive and pulls the plates together, for hydrophobic interfaces it is repulsive and pushes them apart, and at the neutral angle its magnitude is about three orders of magnitude smaller than in these two cases, since only the perimeter term then contributes to the energy. For the hydrophilic interface (panel~(a) of Figure~\ref{fig:capillary-force}) the force grows as the plates approach, exceeding $2\,\gamma_\mathrm{lv} L$ at the smallest separation, and decays monotonically as the liquid bridges thin on retraction; the hydrophobic interface (panel~(c)) is its near mirror image, repulsive and of comparable magnitude.

The neutral case shown in panel~(b) is qualitatively different. Here, at~$\sigma = 0$, the bulk wetting term vanishes, so the force is not an $O(1)$~effect but the first-order, perimeter-only contribution~$\varepsilon F^1$. Its magnitude is thus smaller by roughly a factor~$\varepsilon$. Without the bulk driving force, i.e., with~$\theta=90^\circ$, the liquid is organised entirely by the weighted perimeter and coalesces into the single compact droplet of Figure~\ref{fig:comparison}, column~D\@. The associated energy landscape is flat at leading order and corrugated only by the roughness of the weighted perimeter, so it admits many near-degenerate metastable configurations. As the plates move, the droplet pins and depins on the surface asperities: the contact line locks onto an asperity until the driving force exceeds the local energy barrier, then jumps to the next metastable state. These jumps produce the irregular, avalanche-like behaviour visible in panel~(b). In panels (a) and~(c) these irregularities are not visible because the force is three orders of magnitude larger.

\begin{figure}
    \includegraphics[width=1\linewidth]{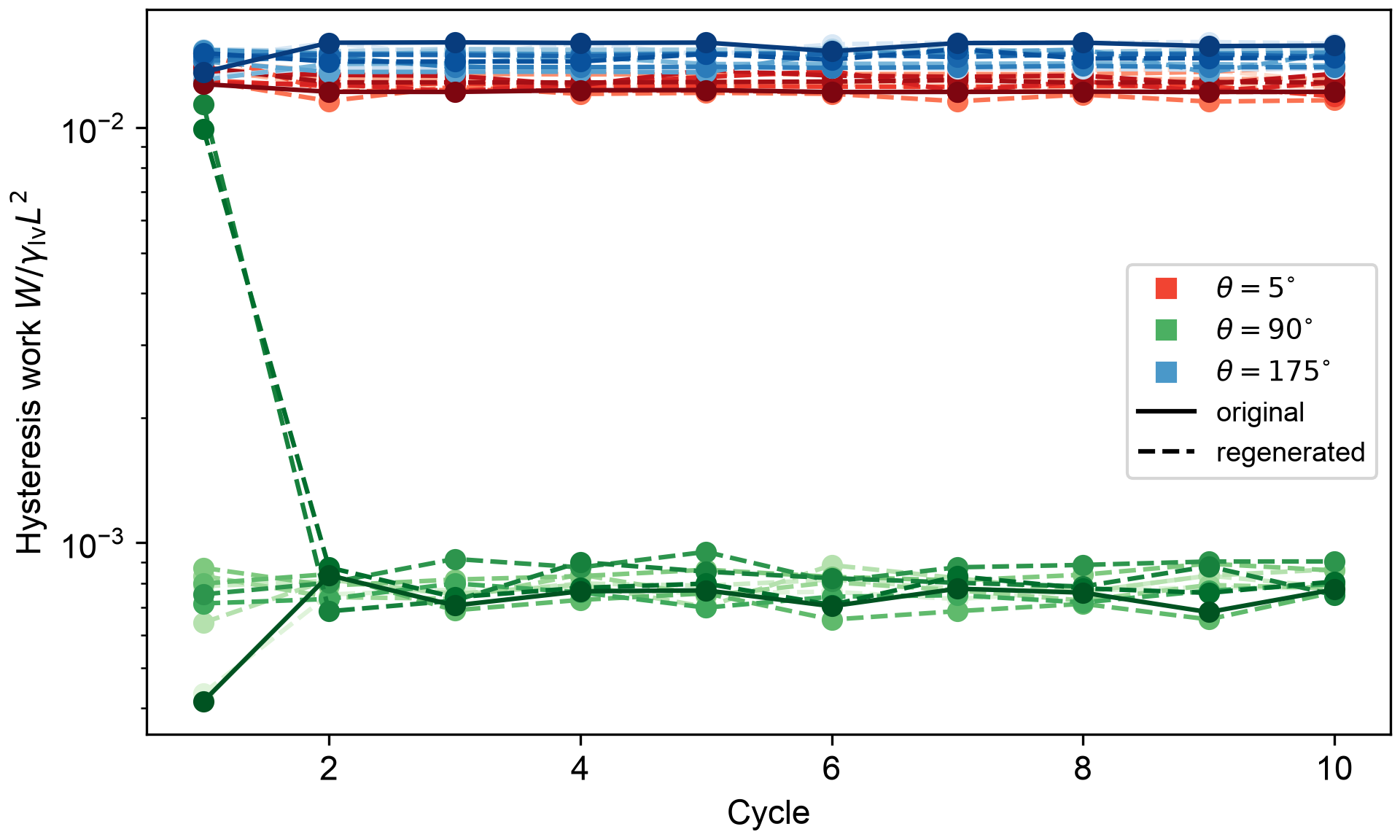}
    \caption{\textbf{Hysteresis work.} The plot shows the hysteresis work computed by numerically integrating the capillary forces over the separation trajectory in Figure~\ref{fig:capillary-force} for each approach--retraction cycle. This is repeated another 10~times, for 10~pairs of computer-generated rough surfaces with different random seeds. The results from regenerated surfaces are plotted with paler colours and dashed lines.}\label{fig:hysteresis}
\end{figure}

Pinning and depinning leads to hysteresis. This is directly visible in the force--separation curves, which enclose a hysteresis loop whose area measures the work dissipated per cycle, which is clearly nonzero even for the extreme contact angles where individual pinning events are overshadowed by the total adhesive force. The hysteresis is hence narrow relative to the overall force scale and smooth in the wetting- and drying-dominated regimes. This behaviour degenerates into an irregular, pinning-dominated band at the neutral angle $\theta=90^\circ$.

In absolute terms, the dissipated work at the neutral angle is more than an order of magnitude below that at the extreme angles (Figure~\ref{fig:hysteresis}). We attribute this to the amount of interface taking part in the depinning events. Each event releases an energy of order~$\varepsilon$ proportional to the length of contact line involved, and the liquid at the neutral angle forms a single compact droplet with little interface, while at the extreme angles it fragments into many islands whose total contact line is far longer.

The simulation data confirm this picture. Figure~\ref{fig:hysteresis-vs-perimeter} shows the work of hysteresis against the total contact-line length at the middle separation, separately for the approach and the retraction halves of each cycle. The contact line of the single droplet at the neutral angle matches the value $2\sqrt{\pi a/z} \approx 1.8L$ for a circular droplet, while the fragmented configurations at the extreme angles carry roughly an order of magnitude more contact line. The dissipated work per unit contact line varies by only a factor of about three across the three contact angles, while the work itself spans more than an order of magnitude. The contact line is itself hysteretic. At the same separation, the retraction states carry roughly one and a half times the contact line of the approach states.

\begin{figure}
    \includegraphics[width=1\linewidth]{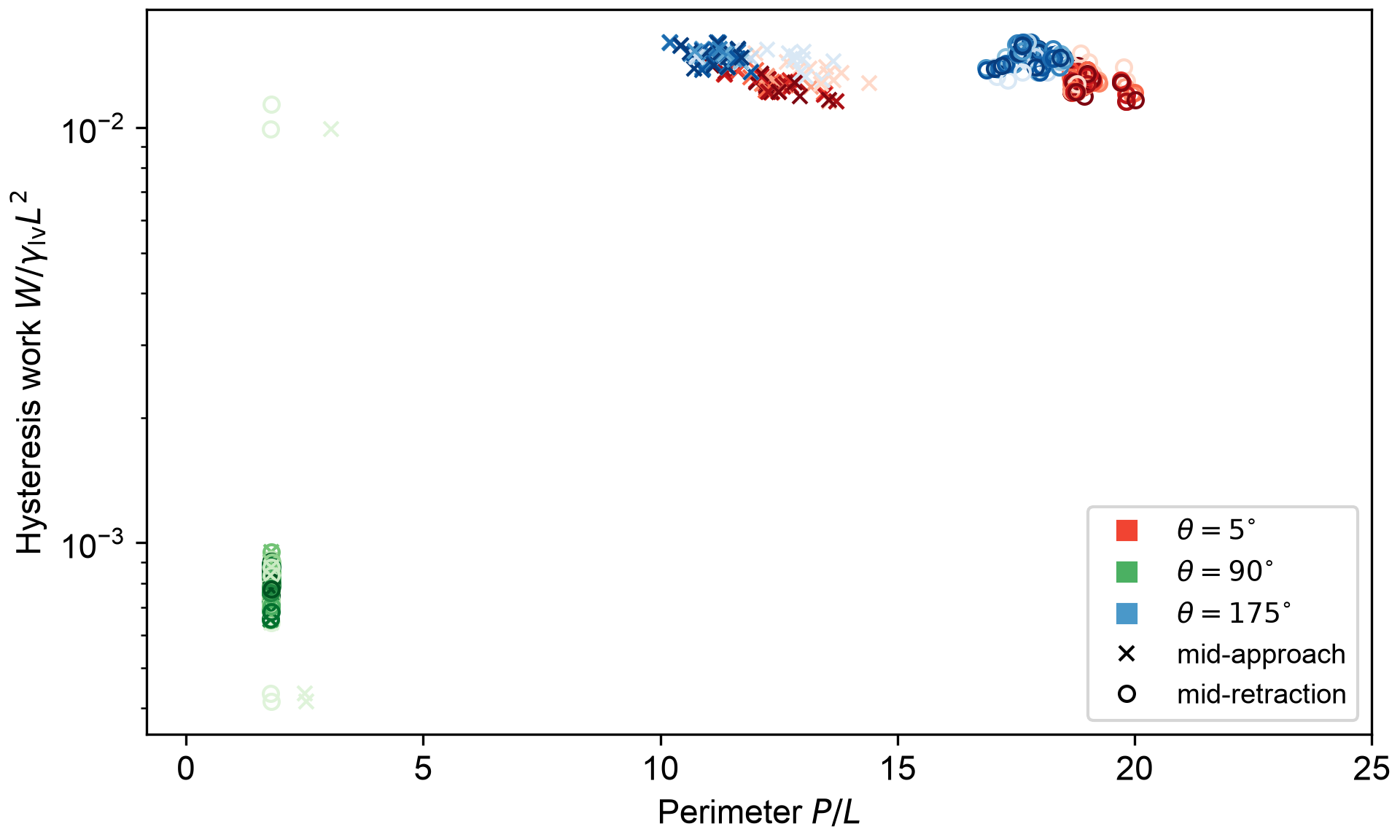}
    \caption{\textbf{Hysteresis work against perimeter.} The plot shows the hysteresis work from Figure~\ref{fig:hysteresis} against the total perimeter of all contact lines. For each approach--retraction cycle, the perimeter at the middle separation ($z=1.0L$) is plotted for both the approach half and the retraction half. Each cycle therefore appears twice, at the same value of the hysteresis work. This also includes the data points from the same regenerated surfaces as in Figure~\ref{fig:hysteresis}.}\label{fig:hysteresis-vs-perimeter}
\end{figure}

The remaining variation of the work per unit contact line is systematic. It increases from the neutral to the hydrophilic to the hydrophobic case. We interpret this ordering as that of the energy released per depinning event. It is smallest at the neutral angle, where the events release no bulk wetting energy, and it grows with the local gap~$g$ at the contact line.

While the energy~\eqref{eq:phase-field-approx-energy} is symmetric between the hydrophilic and the hydrophobic case up to a sign flip of~$\sigma$ and an exchange of liquid and vapour, the hysteresis reveals an asymmetry. The asymmetry arises because the phase field occupies different regions of the gap: the smallest gaps in the hydrophilic case ($\theta<90^\circ$) and the largest gaps in the hydrophobic case ($\theta>90^\circ$). The hydrophobic simulations therefore sample an effective gap larger than that in the hydrophilic simulations. We speculate that the following mechanism produces the offset between the two extreme angles: the barrier for nucleating or annihilating an island scales quadratically with the local line tension~$\varepsilon\kappa_\sigma g$ and inversely with the bulk wetting force, which is weaker for larger gaps. If the energy released per event follows the barrier scale, the moderate difference between the sampled gap levels is amplified nonlinearly.

The data are consistent with this speculation. The contact-line lengths of the hydrophilic and hydrophobic configurations are comparable, with the hydrophilic ones slightly longer, yet the hydrophobic case dissipates more per unit contact line by a factor of approximately~$1.5$ (Figure~\ref{fig:hysteresis-vs-perimeter}). The asymmetry therefore does not lie in the amount of interface but in the energy scale of the individual depinning events. A quantitative analysis of the statistics of these events is beyond the scope of this work and will be the subject of a follow-up study.

\section*{Acknowledgements}

This work was funded by the Deutsche Forschungsgemeinschaft (DFG) within the priority programme SPP~2256, project~523956128. P.~Dondl and L.~Sciaraffia are grateful for inspiring discussions with M.~Novaga (Pisa). L.~Pastewka and Y.~Wang acknowledge fruitful discussions with T.~B.~D.~Jacobs (Pittsburgh) and M.~Ladecký (Freiburg).

\appendix

\section{The notion of \texorpdfstring{$\Gamma$}{Γ}-convergence}

Although the general notion of $\Gamma$-convergence can be developed on any Hausdorff topological space, we use a definition that is suitable only for spaces satisfying the first axiom of countability. Metric spaces fall into this category.

\begin{definition}[$\Gamma$-convergence, cf.~\cite{dalmaso:1993}*{Proposition~8.1}]\label{def:Gamma-lim}
    Let $X$ be a metric space, and let $F_\varepsilon \colon X \to \overline\R$ be a family of functionals, $\varepsilon > 0$. We say that $\{ F_\varepsilon \}_{\varepsilon>0}$ $\Gamma$-converges to $F \colon X \to \overline\R$ as $\varepsilon \downto 0$ if the following conditions are satisfied:
    \begin{itemize}
        \item for every $x \in X$ and for every sequence $\{ x_\varepsilon \}_{\varepsilon>0}$ converging to~$x$
            \[
                F(x) \leq \liminf_{\varepsilon \downto 0} F_\varepsilon(x_\varepsilon) ;
            \]
        \item for every $x \in X$ there exists a sequence $\{ x_\varepsilon \}_{\varepsilon>0}$ converging to~$x$ such that
            \[
                F(x) = \lim_{\varepsilon \downto 0} F_\varepsilon(x_\varepsilon) .
            \]
    \end{itemize}
    We write $F = \Gammalim_{\varepsilon\downto0} F_\varepsilon$.
\end{definition}

\begin{theorem}[Convergence of minimisers I, cf.~\cite{dalmaso:1993}*{Corollary~7.20}]\label{thm:Gamma-min1}
    Assume that $\{ F_\varepsilon \}_{\varepsilon>0}$ $\Gamma$-converges to~$F$ in a metric space~$X$, and let $x_\varepsilon \in X$ be a minimiser of~$F_\varepsilon$ for every $\varepsilon > 0$. If $x$ is a limit point of $\{ x_\varepsilon \}_{\varepsilon>0}$, then $x$ is a minimiser of $F$ in~$X$, with
    \[
        F(x) = \limsup_{\varepsilon \downto 0} F_\varepsilon(x_\varepsilon) .
    \]
    Moreover, if $x = \lim_{\varepsilon \downto 0} x_\varepsilon$, then
    \[
        F(x) = \lim_{\varepsilon \downto 0} F_\varepsilon(x_\varepsilon) .
    \]
\end{theorem}

Let $\{ F_\varepsilon \}_{\varepsilon>0}$ be a family of functionals in~$X$ such that
\[
    F^0 \coloneq \Gammalim_{\varepsilon \downto 0} F_\varepsilon , \qquad m_0 \coloneq \min_{x \in X} F^0(x) < \infty ,
\]
and set $X^0 \coloneq \argmin F^0$.

\begin{definition}[$\Gamma$-expansion, cf.~\cite{anzellotti-baldo:1993}*{Definition~1.3}]\label{def:Gamma-exp}
    We say that the first-order asymptotic development
    \[
        F_\varepsilon = F^0 + \varepsilon F^1 + o(\varepsilon)
    \]
    holds, if
    \[
        \Gammalim_{\varepsilon \downto 0} \frac{ F_\varepsilon - m_0 }{ \varepsilon } = F^1 \quad \text{in} \quad X^0 .
    \]
\end{definition}

It is possible to show that $F^1(x) = \infty$ for all $x \in X \setminus X^0$. In this sense, $F^1$ is \emph{concentrated} on the minimisers of the zeroth-order limit~$F^0$.

As with the (zeroth-order) $\Gamma$-limit, the existence of a development in higher order terms gives more information about the asymptotic behaviour of minimisers.

\begin{theorem}[Convergence of minimisers II, cf.~\cite{anzellotti-baldo:1993}*{Theorem~1.2}]\label{thm:Gamma-min2}
    Suppose the family $\{ F_\varepsilon \}_{\varepsilon>0}$ has a first-order asymptotic development, and let $\{ x_\varepsilon \}_{\varepsilon>0}$ be a family of respective minimisers converging to $x \in X$. Then, $x \in X^0$ and it minimises $F^1$ in $X^0$. Moreover, if $m_\varepsilon$ is the infimum of $F_\varepsilon$ on $X$ and $m_1$ is the infimum of $F^1$ on $X^0$, it holds
    \[
        m_\varepsilon = m_0 + \varepsilon m_1 + o(\varepsilon) .
    \]
\end{theorem}

Finally, we provide the companion notion of $\Gamma$-equivalence.

\begin{definition}[$\Gamma$-equivalence, cf.~\cite{braides-truskinovsky:2008}*{Definition~4.2}]\label{def:Gamma-equiv}
    Let $\{ F_\varepsilon \}_{\varepsilon>0}$ and $\{ G_\varepsilon \}_{\varepsilon>0}$ be two families of functionals on a metric space~$X$. We say that they are equivalent at order~$\varepsilon^\beta$ ($\beta>0$) if there exist $m_\varepsilon \in \R$ such that for all sequences $\varepsilon_k \downto 0$ for which the limits exist we have
    \[
        \Gammalim_{k \to \infty} \frac{ F_{\varepsilon_k} - m_{\varepsilon_k} }{ \varepsilon_k^\beta } = \Gammalim_{k \to \infty} \frac{ G_{\varepsilon_k} - m_{\varepsilon_k} }{ \varepsilon_k^\beta }
    \]
    and the limits are nontrivial (i.e.~they do not take the value~$-\infty$ and are not identically~$\infty$).
\end{definition}

\section{The Minkowski dimension}

Although there are many equivalent definitions of the Minkowski dimension found in the literature, we provide the two that are of use to us in the construction of Theorem~\ref{thm:irreg-mins}.

\begin{definition}[Minkowski dimension, cf.~\cite{falconer:2014}*{Definitions~2.1, Proposition~2.4}]\label{def:minkowski}
    Let $X \subset \R^d$. For $\delta > 0$, let $N_\delta(X)$ be the largest number of disjoint balls~$B_\delta(x)$ with centres~$x \in X$. Then, the \emph{Minkowski dimension} of~$X$ is given by
    \[
        \alpha \coloneq \lim_{\delta \downto 0} \frac{ \log N_\delta(X) }{ - \log\delta } ,
    \]
    provided this limit exists. Equivalently, if~$B_\delta(X)$ denotes the set of points at distance less than $\delta$ from~$X$, then~$\alpha$ can be computed as
    \[
        \alpha = d - \lim_{\delta \downto 0} \frac{\log |B_\delta(X)| }{ \log \delta } .
    \]
\end{definition}

\end{document}